\documentclass[10pt,a4paper,reqno]{amsart}            

\usepackage[english]{babel}
\usepackage{amssymb,amsmath,amsthm}
\usepackage{mathrsfs}

\usepackage{enumitem}
\setlist[enumerate]{leftmargin=*}

\usepackage{hyperref}

\renewcommand\mathcal{\mathscr}

\theoremstyle{plain}
\newtheorem{theorem}{Theorem}[section]
\newtheorem{lemma}[theorem]{Lemma}
\newtheorem{corollary}[theorem]{Corollary}
\newtheorem{proposition}[theorem]{Proposition}

\theoremstyle{remark}
\newtheorem{remark}[theorem]{Remark}

\theoremstyle{definition}
\newtheorem{definition}[theorem]{Definition}

\numberwithin{equation}{section}

\newcount\quantno
\everydisplay\expandafter{\the\everydisplay\quantno=0}
\everycr\expandafter{\the\everycr\quantno=0}
\newcommand\quant{\advance\quantno by1
                      \ifnum\quantno=1\qquad\else\quad\fi\forall }

\newcommand{\loc}{\mathrm{loc}}

\newcommand{\compl}{\mathsf{c}}

\renewcommand\Re{\operatorname{\mathrm{Re}}}
\renewcommand\Im{\operatorname{\mathrm{Im}}}

\newcommand\ir{\int_\BR}
\newcommand\ioty{\int_0^{\infty}}
\newcommand\wrt{\,\mathrm{d}}
\newcommand\dtt[1]{\,\frac{\mathrm{d} #1}{ #1}}

\newcommand\e{\mathrm{e}}

\newcommand\bc{\mathbf{c}}

\newcommand\wh{\widehat}
\newcommand\wt{\widetilde}

\newcommand{\defeq}{\mathrel{:=}}

\renewcommand\mod[1]{\vert{#1}\vert}
\newcommand\bigmod[1]{\bigl\vert{#1}\bigr|}

\newcommand\norm[2]{{\Vert{#1}\Vert_{#2}}}

\newcommand\bignorm[2]{\left.{\bigl\Vert{#1}\bigr\Vert_{#2}}\right.}

\newcommand\bigopnorm[2]{\big|\!\big|\!\big| {#1} \big|\!\big|\!\big|_{#2}}

\newcommand\bS{\mathbf{S}} 
\newcommand\bT{\mathbf{T}} 

\newcommand\BC{\mathbb{C}}
\newcommand\BN{\mathbb{N}}
\newcommand\BR{\mathbb{R}}
\newcommand\BZ{\mathbb{Z}}

\newcommand\cE{\mathcal{E}}  
\newcommand\cH{\mathcal{H}}  
\newcommand\frh{\mathfrak{h}} 
\newcommand\cI{\mathcal{I}}  
\newcommand\cL{\mathcal{L}}  
\newcommand\cM{\mathcal{M}}  
\newcommand\cP{\mathcal{P}}  
\newcommand\cR{\mathcal{R}}  
\newcommand\cT{\mathcal{T}}    
\newcommand\cU{\mathcal{U}}  
\newcommand\cX{\mathcal{X}}  
\newcommand\fX{\mathfrak{X}}

\newcommand\al{\alpha}

\newcommand\ga{\gamma}    \newcommand\Ga{\Gamma}
\newcommand\de{\delta}
  
\newcommand\vep{\varepsilon}

\newcommand\la{\lambda}   
    
\newcommand\si{\sigma}    

\newcommand\vp{\varphi}

\newcommand\maxP{\cP_*}
\newcommand\maxPc{\cP_*^c}
\newcommand\maxPloc{\cP_*^{\mathrm{loc}}}
\newcommand\maxH{\cH_*}
\newcommand\maxHc{\cH_*^c}
\newcommand\maxHloc{\cH_*^{\mathrm{loc}}}

\newcommand\lu[1]{L^1(#1)}

\newcommand\hu[1]{H^1(#1)}
\newcommand\huR[1]{H_{\cR}^1(#1)}
\newcommand\huP[1]{H_{\cP}^1(#1)}
\newcommand\huPc[1]{H_{\cP,c}^1(#1)}
\newcommand\huPu[1]{H_{\cP,1}^1(#1)}
\newcommand\huH[1]{H_{\cH}^1(#1)}
\newcommand\huHc[1]{H_{\cH,c}^1(#1)}

\newcommand\ghu[1]{{\frh}^1(#1)}
\newcommand\ghuRtau[1]{{\frh}_{\cR_\tau}^1(#1)}
\newcommand\ghuRtaup[1]{{\frh}_{\cR_{\tau'}}^1(#1)}

\newcommand\gXum[1]{\fX^{1/2}(#1)}
\newcommand\gXu[1]{{\fX}^1(#1)}
\newcommand\gXga[1]{\fX^{\ga}(#1)}

\DeclareSymbolFont{EUEX}{U}{euex}{m}{n}

\DeclareSymbolFont{euexlargesymbols}{U}{euex}{m}{n}
\DeclareMathSymbol{\intop}{\mathop}{euexlargesymbols}{"52}
     \def\int{\intop\nolimits}

\DeclareSymbolFont{euexsymbols}     {U}{euex}{m}{n}
\DeclareMathSymbol{\smallint}{\mathop}{euexsymbols}{"52}

\begin{document}

\title[Comparison of Hardy type spaces]{Inclusions and noninclusions of Hardy type spaces on certain nondoubling manifolds}



\thanks{The first-named author gratefully acknowledges the support of Compagnia di San Paolo.
The first-, third- and fourth-named authors are members of the Gruppo Nazionale per l'Analisi Mate\-matica, 
la Probabilit\`a e le loro Applicazioni (GNAMPA) of the Istituto 
Nazionale di Alta Mate\-matica (INdAM)}

\author[A. Martini]{Alessio Martini}
\address[Alessio Martini]{Dipartimento di Scienze Matematiche ``Giuseppe Luigi Lagrange'', Dipartimento di Eccellenza 2018-2022 \\
Politecnico di Torino\\ corso Duca degli Abruzzi 24\\ 10129 Torino\\ Italy}
\email{alessio.martini@polito.it}

\author[S. Meda]{Stefano Meda}
\address[Stefano Meda]{Dipartimento di Matematica e Applicazioni \\ Universit\`a di Milano-Bicocca\\
via R.~Cozzi 53\\ I-20125 Milano\\ Italy}
\email{stefano.meda@unimib.it}

\author[M. Vallarino]{Maria Vallarino}
\address[Maria Vallarino]{Dipartimento di Scienze Matematiche ``Giuseppe Luigi Lagrange'', Dipartimento di Eccellenza 2018-2022 \\
Politecnico di Torino\\ corso Duca degli Abruzzi 24\\ 10129 Torino\\ Italy}
\email{maria.vallarino@polito.it}

\author[G. Veronelli]{Giona Veronelli}
\address[Giona Veronelli]{Dipartimento di Matematica e Applicazioni \\ Universit\`a di Milano-Bicocca\\
via R.~Cozzi 53\\ I-20125 Milano\\ Italy}
\email{giona.veronelli@unimib.it}

\begin{abstract}
	In this paper we establish inclusions and noninclusions between various Hardy type spaces on noncompact Riemannian manifolds $M$ with
	Ricci curvature bounded from below, positive injectivity radius and spectral gap.  

	Our first main result states that, if $\cL$ is the positive Laplace--Beltrami operator on $M$, then the Riesz--Hardy space $\huR{M}$ is the isomorphic image of the Goldberg type space $\ghu{M}$ via the map $\cL^{1/2} (\cI + \cL)^{-1/2}$,
	a fact that is false in $\BR^n$. 
	Specifically, $\huR{M}$ agrees with the Hardy type space $\gXum{M}$ recently introduced by the the first three authors;
	as a consequence, we prove that $\huR{M}$ does not admit an atomic characterisation.  

	Noninclusions are mostly proved in the special case where the manifold is a Damek--Ricci space $S$. 
	Our second main result states that $\huR{S}$, the heat Hardy space $\huH{S}$ and the Poisson--Hardy space $\huP{S}$ 
	are  mutually distinct spaces, a fact which is in sharp contrast to the Euclidean case, where these three spaces agree.  
\end{abstract}

\maketitle

\section{Introduction}
\label{s: Introduction}
The purpose of this paper is to prove a number of results 
concerning Hardy type spaces on a certain class of nondoubling Riemannian manifolds.  Our results illustrate that the
scenario of Hardy spaces on such class of manifolds may differ considerably from that we are familiar with on Euclidean spaces.  

Hardy spaces have become a landmark in the panorama of Euclidean Harmonic Analysis (in several variables) after the 1972 seminal paper of C.~Fefferman and E.~M.\ Stein \cite{FS}.
It is virtually impossible in a research paper to give appropriate credit to all the mathematicians who have contributed to develop
the theory of Hardy spaces in an impressive variety of settings besides $\BR^n$.  Some pointers on the existing literature may be found in the
introduction of \cite{MaMV1}, to which we refer the interested reader.  

One of the key features of $\hu{\BR^n}$ is its ``flexibility'', which is a consequence of its many different characterisations: these
include the atomic and the maximal characterisations and the characterisation via Riesz transforms.
This makes it possible to choose the most useful characterisation of $\hu{\BR^n}$ in connection with a specific application one has in mind. 

Suppose that $M$ is a complete connected noncompact Riemannian manifold.  Denote by $\cL$ the (nonnegative) Laplace--Beltrami operator on $M$, 
by $\nabla$ the covariant derivative.  Consider the \emph{Riesz transform} $\cR \defeq \nabla \cL^{-1/2}$ and the following Hardy type spaces,
endowed with their natural norms:
\begin{enumerate}[label=(\roman*)]
	\item
		the \emph{Riesz--Hardy space} $\huR{M}$, defined by 
		\begin{equation}\label{f: rieszhardy}
		\huR{M}
		\defeq \{f \in \lu{M}: \mod{\cR f} \in \lu{M}\};
		\end{equation}
	\item
		the \emph{heat maximal Hardy space} $\huH{M}$, associated to the heat semigroup $\{\cH_t\defeq\e^{-t\cL}: t\geq 0\}$, 
		defined by 
		\begin{equation}\label{f: heathardy}
		\huH{M}
		\defeq \big\{f \in \lu{M}: \maxH f \in \lu{M} \big\},
		\end{equation}
		where $\maxH f \defeq \sup_{t > 0} \, \bigmod{\cH_t f}$; 
	\item
		the \emph{Poisson maximal Hardy space} $\huP{M}$, associated to the Poisson semigroup $\{\cP_t\defeq\e^{-t\cL^{1/2}}: t\geq 0\}$, 
		defined by 
		\begin{equation}\label{f: heatPoisson}
		\huP{M}
		\defeq \big\{f \in \lu{M}: \maxP f \in \lu{M} \big\},
		\end{equation}
		where $\maxP f \defeq \sup_{t > 0} \, \bigmod{\cP_t f}$. 
\end{enumerate}
As mentioned above (see \cite{FS})
\begin{equation} \label{f: FS equality}
\huR{\BR^n} = \huH{\BR^n} = \huP{\BR^n}.  
\end{equation}  
It is natural to speculate whether a similar equality holds in wider generality.  A discussion of this interesting problem may be found in
the introduction of \cite{MaMV1}, to which the reader is referred for further information.  
Here we content ourselves to mention that, as a consequence of the efforts of various authors \cite{AMR,DKKP,HLMMY}, 
\[
\huR{M} \supseteq \huH{M} = \huP{M}  
\]
in the case where $M$ is a \textbf{doubling} Riemannian manifold with Ricci curvature bounded from below and positive injectivity radius.  
The first inclusion is a trivial consequence of the boundedness of the Riesz transform $\cR$ from, say, $\huH{M}$ to $\lu{M}$.  To the best of our
knowledge, it is unknown whether, under the assumptions above, $\huR{M} = \huH{M}$.

We emphasize the fact that there are only a few (very specific) examples in the literature, besides $\BR^n$, 
in which the Riesz--Hardy space $\huR{M}$ is characterised in whatsoever form.  
For more on this, see the introduction of \cite{MaMV1} and the references therein.  

In this paper we are interested exclusively in analogues of the space $\hu{\BR^n}$ on manifolds in the class $\cM$
of all complete noncompact connected Riemannian manifolds $M$ satisfying the following:
\begin{enumerate}[label=(\roman*)]
\item
the injectivity radius of $M$ is positive;
\item
the Ricci tensor of $M$ is bounded from below; 
\item
$M$ has spectral gap, i.e., the bottom $b$ of the $L^2$-spectrum of the Laplace--Beltrami operator $\cL$ is strictly positive. 
\end{enumerate}
Riemannian manifolds in $\cM$ are nondoubling metric measure spaces.  Notice that~$\cM$ includes all symmetric spaces 
of the noncompact type and Damek--Ricci spaces, as well as all the simply connected complete Riemannian manifolds with negative pinched sectional
curvature such as, for instance, the universal coverings of compact manifolds of strictly negative curvature.
Note that $\cM$ is precisely the class of manifolds considered in \cite{MaMV1}.  
We refer to \cite[pp.\ 2064--2065]{MaMV1} for some important analytic and geometric consequences of the assumptions above.  

Many variants of Hardy type spaces have been considered on (subclasses of) the class $\cM$.  Each of them is taylored
to obtain endpoint estimates for certain classes of operators.   Without any pretence of exhaustiveness, we mention
\cite{A1,Lo,I,MM,CMM1,CMM2,MMV1,MMV2,MMV3,MMV4,MMV5,MaMV1,MaMV2,MVe,T}, and refer the reader
to the introductive sections of these papers for more on Hardy type spaces and their role in obtaining sharp estimates
for a variety of operators, and for pointers to the vast literature on the subject.

One of the main results of this paper, Corollary~\ref{c: comparison heat Poisson}, is the perhaps surprising fact 
that in the class $\cM$ the ideal chain of equalities \eqref{f: FS equality} may fail dramatically.   
Specifically, it asserts that, if $S$ is a Damek--Ricci space, then $\huR{S}$, $\huH{S}$
and $\huP{S}$ are pairwise distinct spaces.  Notice that this is in sharp contrast with the Euclidean case discussed above. It is likely that a similar phenomenon happens on more general manifolds than Damek--Ricci spaces. As a matter of fact, we can rule out that the spaces $\huR{M}$ and $\huP{M}$ coincide when $M$ is an  arbitrary manifold in the class $\cM$ (see Corollary \ref{c: maximal heat manifold}).

In order to effectively compare $\huR{M}$, $\huH{M}$ and $\huP{M}$, 
we make use of two main ideas: the first is to consider 
certain variants of $\huH{M}$ and $\huP{M}$,
the second is to characterise $\huR{M}$ as an isomorphic image of the local Hardy space $\ghu{M}$ of Goldberg type on $M$
(we emphasize that an analogous characterisation fails
 in $\BR^n$, 
see Proposition~\ref{p: hu ghu Rn});
see Section~\ref{s: Hardy--Riesz space} for details on $\ghu{M}$.  
We offer the following comments concerning the genesis of these two ideas.  

The first is suggested by the reading of \cite{A1}, where J.-Ph.~Anker sharpened previous results of N.~Lohou\'e \cite{Lo} 
and proved that if $M$ is a symmetric space of the noncompact type, then for every $c$ in $[0,1)$ the following estimate holds:  
\[
	\bignorm{\cP^c_* f}{\lu{M}} 
	\leq C \, \big(\bignorm{f}{\lu{M}} + \bignorm{\cR f}{\lu{M}}\big).
\]
Here 
\[
\cP^c_* f
\defeq \sup_{t>0} \, \langle t\rangle^c \, \bigmod{\cP_t f},
\]
and $\langle t\rangle =\max\{1,t\}$.  In other words, $\huR{M} \subseteq \huP{M}$ (and, more generally, $\huR{M} \subseteq \huPc{M}$ for 
all $c\in [0,1)$ in the notation of Section~\ref{s: Maximal}).  This surprising result strongly suggests that it may be advantageous
to consider a finer scale of spaces $\huPc{M}$, including $\huP{M}$, defined in terms of the modified maximal function $\cP_*^c$ above.  
By analogy, it seems natural to consider a similar scale of spaces $\huHc{M}$, including $\huH{M}$, 
defined in terms of the modified maximal function $\cH_*^c$, defined as follows:
\[
\cH^c_* f
\defeq \sup_{t>0} \, \langle t\rangle^c \, \bigmod{\cH_t f}.
\]
The scales of spaces $\bigl\{\huPc{M}: c \in \BR\bigr\}$ and $\bigl\{\huHc{M}: c \in \BR\bigr\}$ will play a fundamental role in what follows
(see Section~\ref{s: Maximal}). 

The second idea was, in fact, one of the motivations behind the introduction and the study in \cite{MaMV1} of the scale 
of Hardy type spaces $\{\gXga{M}: \ga >0\}$ (see Definition~\ref{def: gXga}), which, by definition, are isometric images of the Goldberg 
type space $\ghu{M}$ via the fractional powers $\cU^\ga$ of the operator $\cU = \cL(\cI+\cL)^{-1}$.  In \cite{MaMV1} it was proved
that for manifolds in the class $\cM$ the space $\gXum{M}$ agrees with $\{f\in \ghu{M}: \mod{\cR f} \in \lu{M}\}$,
which, in principle, may be strictly contained in $\huR{M}$.  In this paper, armed with the Riesz transform characterisation of $\ghu{M}$ recently obtained in \cite{MVe}, we improve this result and show that 
the Riesz--Hardy space $\huR{M}$ agrees with the Banach space $\gXum{M}$: see Theorem~\ref{t: char Riesz}. 
We emphasize that this result justifies retrospectively the introduction of the space $\gXum{M}$.  Thus,
for every $M$ in the class $\cM$, the Riesz--Hardy space $\huR{M}$ turns out to be an isometric copy of $\ghu{M}$ via the map $\cU^{1/2}$,
a fact that fails in the Euclidean setting
(see Proposition~\ref{p: hu ghu Rn}).  

The fact that if $S$ is a Damek--Ricci space, then $\huR{S}$, $\huH{S}$ and $\huP{S}$ are mutually distinct spaces 
is a consequence of fine estimates of the kernels of certain operators on Damek--Ricci spaces and of inclusions between the 
families of Hardy type spaces $\huHc{M}$ and $\huPc{M}$ for arbitrary manifolds $M$ in the class $\cM$.  Our analysis shows that 
the containments amongst the spaces $\huHc{S}$ and $\huPc{S}$ and the abovementioned family of spaces $\gXga{S}$ are subtle and
highly nontrivial.  In particular, we show that $\gXum{S}$ is very similar to but not
quite the same as $\huPu{S}$: see Corollary \ref{c: maximal heat manifold} and Theorem \ref{t: comparison weighted poisson} below.

It is interesting to observe that $\gXum{M}$ is the space of all functions of the form 
$\cL^{1/2}(\cI+\cL)^{-1/2}f$, where $f$ belongs to the Goldberg type space $\ghu{M}$.  Amongst the direct consequences of this fact, we mention
the following.  Proving that the spectral operator $m(\cL)$ maps $\huR{M}$ to $\lu{M}$ is equivalent
to showing that the spectral operator $m(\cL) \cL^{1/2}(\cI+\cL)^{-1/2}$ maps $\ghu{M}$ to $\lu{M}$.  This does not involve estimates of gradients
any more and can hopefully be done using spectral methods.  

We mention that the results contained in this paper are versions in the continuous setting of
similar results obtained in the discrete case of graphs and trees \cite{CM}.
It is fair to observe that analysis at local scales, which turns out to be trivial on graphs, is highly nontrivial in our situation. 
Notice also that local and global analysis on manifolds cannot be neatly separated, and one influences the other.  As a consequence, our methods
turn out to be comparatively more involved than those in \cite{CM}.  For instance, our first result (concerning the abovementioned
characterisation of the Riesz--Hardy space), which hinges on the detailed analysis performed in \cite{MVe} and on spectral methods,
is rather trivial on graphs.  

\medskip

Our paper is organised as follows.  Section~\ref{s: Hardy--Riesz space} contains our first main result, Theorem~\ref{t: char Riesz}, together with
an important improvement (Proposition \ref{p: improvement tau}) of the main result in \cite{MVe}.  In Section~\ref{s: Maximal}, 
we introduce the spaces $H^1_{\cH,c}(M)$ and $H^1_{\cP,c}(M)$,
defined in terms of certain weighted heat and Poisson maximal operators, respectively.  We prove some inclusions
amongst $H^1_{\cH,c}(M)$, $H^1_{\cP,c}(M)$ and $\gXga{M}$ on any manifold $M$ in the class $\cM$.  Our methods, based on functional calculus, are 
comparatively simple and direct.  Among other things (see Corollary \ref{c: maximal heat manifold}), we show that  
\[
H_{\cP,2\ga+\vep}^1(M) \subset \gXga{M} \subset H_{\cP,2\ga-\vep}^1(M)
\]
for every $\vep>0$ and $\ga>0$.  A special case of
interest of this chain of inclusions is 
\[
H_{\cP,1+\vep}^1(M) \subset \huR{M} \subset H_{\cP,1-\vep}^1(M)
\]
for every $\vep>0$, which implies in particular that $\huR{M}$ is properly contained in $\huP{M}$.  Concerning the relationship between
$H_{\cH,c}^1(M)$ and $\gXga{M}$ we are only able to show, for an arbitrary manifold $M$ in the class $\cM$, that
\[
H_{\cH,\ga+\vep}^1(M) \subset \gXga{M}
\]
for every $\vep>0$. We can prove an inclusion in the opposite direction in the special case where the manifold is a Damek--Ricci space, where spherical Fourier analysis allows us to obtain precise estimates of the heat kernel; this is discussed in Section~\ref{s: Maximal heat operator}, which is mainly devoted to proving noninclusions amongst $\huR{S}$, $\huH{S}$ and $\huP{S}$,
where $S$ is a Damek--Ricci space.  As a consequence, we show that $\huR{S}$, $\huH{S}$ and $\huP{S}$ are
mutually distinct spaces (see Corollary~\ref{c: comparison heat Poisson}).

\medskip

In the course of the paper, the letter $C$ is used to denote a finite positive constant, whose value may change from place to place. We also write $A \asymp B$ to indicate that $A \leq C B$ and $B \leq C A$. The symbols $\subseteq$ and $\subset$ denote set inclusion and proper set inclusion respectively, while $S^\compl$ denotes the complement of the set $S$. For a real number $x$, we write $\lceil x \rceil$ to denote the smallest integer greater than or equal to $x$.

\section{A characterisation of the Riesz--Hardy space}
\label{s: Hardy--Riesz space}

In this section we prove our first main result, Theorem~\ref{t: char Riesz}, about the characterisation of the Riesz--Hardy space on a manifold $M$ of class $\cM$.  As mentioned in the introduction, 
this result highlights a relationship between the local Hardy space $\ghu{M}$ and the Riesz--Hardy space $\huR{M}$ that has no counterpart in the Euclidean setting. 
In order to clearly illustrate this, we shall briefly discuss the Euclidean case (see Proposition \ref{p: hu ghu Rn} below) before moving to the case of manifolds.
First of all, however, as the methods in this section are based on functional calculus techniques, some preliminary considerations on the matter are in order.

\subsection{The extended Dunford class}
If $\theta \in (0,\pi)$, by $\bS_\theta$ we denote the open sector in the complex plane $\BC$ with aperture $2\theta$ symmetric about the positive real axis. We recall that the \emph{extended Dunford class} $\cE(\bS_\theta)$ is the space of all bounded holomorphic functions $F : \bS_\theta \to \BC$ for which
there exist $w_0,w_\infty \in \BC$ and $\varepsilon > 0$ such that
\begin{equation}\label{f: Dunford limit}
|F(z) - w_0| + |F(1/z)-w_\infty| = O(|z|^\varepsilon) \quad\text{ as } z \to 0
\end{equation}
(see \cite[Lemma 2.2.3]{Haa}). We record here some elementary properties of $\cE(\bS_\theta)$ that will be of use in the sequel (cf.\ \cite[Examples 2.2.4 and 2.2.5]{Haa}).

\begin{lemma}\label{l: Dunford class}
Let $\theta \in (0,\pi)$.
\begin{enumerate}[label=(\roman*)]
\item\label{en:dunford_algebra} $\cE(\bS_\theta)$ is a unital algebra under pointwise addition and multiplication.
\item\label{en:dunford_reciprocal} If $F$ is in $\cE(\bS_\theta)$ and $1/F$ is a bounded holomorphic function on $\bS_\theta$, then $1/F$ is in $\cE(\bS_\theta)$.
\item\label{en:dunford_examples} Any function of the form
\begin{equation}\label{f: dunford_example}
 z \mapsto \left(\frac{a + bz}{c+dz}\right)^\gamma,
\end{equation}
with $\gamma,a,b \in [0,\infty)$ and $c,d \in (0,\infty)$, is in $\cE(\bS_\theta)$.
\end{enumerate}
\end{lemma}
\begin{proof}
Part \ref{en:dunford_algebra} is discussed in \cite[Section 2.2]{Haa}.

As for part \ref{en:dunford_reciprocal}, if $F \in \cE(\bS_\theta)$ and $1/F$ is bounded, then the limits $w_0,w_\infty \in \BC$ of $F$ at $0$ and $\infty$ must be nonzero. Consequently
\[
\left|\frac{1}{F(z)}-\frac{1}{w_0}\right| + \left|\frac{1}{F(1/z)}-\frac{1}{w_\infty}\right|= \frac{|F(z)-w_0|}{|F(z)| |w_0|}+ \frac{|F(1/z)-w_0|}{|F(1/z)| |w_\infty|} = O(|z|^\varepsilon)
\]
as $z \to 0$, for some $\varepsilon > 0$, where we used \eqref{f: Dunford limit} and the boundedness of $1/F$.

Finally, for part \ref{en:dunford_examples}, any function $F$ of the form \eqref{f: dunford_example} is bounded and holomorphic on $\bS_\theta$, and has finite limits $w_0 = (a/c)^\gamma$ and $w_\infty = (b/d)^\gamma$ as $z \to 0$ and $z \to \infty$ respectively. If $w_0 \neq 0$, then $F$ extends to a holomorphic function in a neighbourhood of $0$, thus clearly $|F(z) - w_0| = O(|z|)$ as $z \to 0$. If instead $w_0 = 0$, then $a=0$ and $|F(z)| = O(|z|^\gamma)$ as $z \to 0$. Similar considerations apply to $F(1/z)$ and $w_\infty$ in place of $F(z)$ and $w_0$, thus showing that $F$ is in $\cE(\bS_\theta)$.
\end{proof}

The relevance for our discussion of the extended Dunford class lies in its role in the functional calculus for sectorial operators, as described, e.g., in \cite[Chapter 2]{Haa}.
Specifically:
\begin{itemize}
\item if $\cT$ is a sectorial operator of angle $\omega \in [0,\pi)$ on a Banach space $\cX$ (see, e.g., \cite[Section 2.1]{Haa}), then, for all $\theta \in (\omega,\pi)$ and $F \in \cE(\bS_\theta)$, the operator $F(\cT)$ is bounded on $\cX$ \cite[Theorem 2.3.3]{Haa};
\item if $\cT$ is the generator of a uniformly bounded semigroup of class $C_0$ on $\cX$, then $\cT$ is sectorial of angle $\pi/2$ (see \cite[Section 2.1.1, p.\ 24]{Haa}).
\end{itemize}
These facts will be repeatedly used in what follows.

\subsection{Hardy spaces in the Euclidean setting}
The local Hardy space $\ghu{\BR^n}$ was introduced in D.~Goldberg's paper \cite{Go}.  Since then, $\ghu{\BR^n}$ is commonly
referred to as the Goldberg space.  
We start by establishing a result that relates the classical Hardy space $\hu{\BR^n}=\huR{\BR^n}$ and the Goldberg space $\ghu{\BR^n}$.  
Consider the 
operator $\cL^{1/2}(\cI+\cL)^{-1/2}$, also denoted by $\cU^{1/2}$, where $-\cL = \Delta$ denotes the standard Laplacian. 

\begin{proposition} \label{p: hu ghu Rn}
The operator $\cU^{1/2}$ is bounded and injective, but not surjective, from $\ghu{\BR^n}$ to $\hu{\BR^n}$. Thus, the space $H^1(\BR^n)$ contains $\cU^{1/2} [\ghu{\BR^n}]$ properly.
\end{proposition}

\begin{proof}
First we show that if $g$ is in $\ghu{\BR^n}$, then $\cU^{1/2} g$ is in $\hu{\BR^n}$.  Since $\cL$ generates a contractive semigroup
of class $C_0$ on $\lu{\BR^n}$ and the function 
\[
\vp(z) \defeq \left(\frac{z}{1+z}\right)^{1/2}
\]
is in the extended Dunford class
$\cE(\bS_\theta)$ for every $\theta$ in $(\pi/2,\pi)$ (see Lemma \ref{l: Dunford class}), 
the operator $\cU^{1/2} = \vp(\cL)$ is bounded on $\lu{\BR^n}$ (see \cite[Theorem 2.3.3]{Haa}).  Hence $\cU^{1/2}g$
is in $\lu{\BR^n}$.  In order to show that $\cU^{1/2}g$ is in $\hu{\BR^n}$, it suffices to prove that 
$\mod{\cR\cU^{1/2} g}$ is in $\lu{\BR^n}$.  Notice that 
\[
	\mod{\cR\cU^{1/2} g} = \mod{\nabla \cL^{-1/2} \cU^{1/2} g} = \mod{\nabla (\cI+\cL)^{-1/2} g},   
\]
which belongs to $\lu{\BR^n}$, because $g$ belongs to $\ghu{\BR^n}$, by assumption. 
A close examination of the argument above shows that there exists a constant $C$ such that 
\[
	\bignorm{\cU^{1/2} g}{\hu{\BR^n}} 
	\leq C\, \big(1+\bigopnorm{\cU^{1/2}}{\lu{\BR^n}}\big) \, \bignorm{g}{\ghu{\BR^n}}
	\quant g \in \ghu{\BR^n}.
\]

Note that $\cU^{1/2}$ is injective on $\ghu{\BR^n}$.  Indeed, if $\cU^{1/2}g = 0$, then
\[
\frac{\mod{\xi}}{\sqrt{1+\mod{\xi}^2}} \, \wh g(\xi) = 0  
\quant \xi \in \BR^n.  
\]
Therefore $\wh g(\xi)$ vanishes on $\BR^n\setminus\{0\}$, hence everywhere, for $\wh g$ is continuous.  Thus, $g=0$, as required. 

Finally, we show that $\cU^{1/2} \left[\ghu{\BR^n}\right]$ is properly contained in $\hu{\BR^n}$.  We argue by contradiction.  If $\cU^{1/2}$
were surjective, then $\cU^{-1/2}$ would be bounded from $\hu{\BR^n}$ to $\ghu{\BR^n}$, and in particular from $\hu{\BR^n}$ to $\lu{\BR^n}$.
We prove that this fails.
Indeed, $(\cI+\cL)^{-1/2} \cU^{-1/2} = \cL^{-1/2}$, and $(\cI+\cL)^{-1/2}$ is bounded on $\lu{\BR^n}$. Thus, if $\cU^{-1/2}$ were bounded from $\hu{\BR^n}$ to $\lu{\BR^n}$, the same would be true of $\cL^{-1/2}$. However, the boundedness of $\cL^{-1/2}$ from $\hu{\BR^n}$ to $\lu{\BR^n}$ is easily disproved by homogeneity considerations.

Indeed, for any $R > 0$, the formula $T_R f(x) = R^{-n} f(x/R)$ defines an isometric automorphism $T_R$ of both $\lu{\BR^n}$ and $\hu{\BR^n}$. On the other hand, by homogeneity of $\cL = -\Delta$ on $\BR^n$, we deduce that $\cL^{-1/2} T_R f = R \, T_R \cL^{-1/2} f$; thus, for any nonzero $f \in \hu{\BR^n}$,
\[
\| \cL^{-1/2} T_R f \|_{\lu{\BR^n}} = R \| \cL^{-1/2} f \|_{\lu{\BR^n}} \to \infty \qquad\text{as } R \to \infty,
\]
while
\[
\| T_R f \|_{\hu{\BR^n}} = \| f \|_{\hu{\BR^n}} \quant R > 0,
\]
whence $\cL^{-1/2}$ cannot be bounded from $\hu{\BR^n}$ to $\lu{\BR^n}$.
\end{proof}

\subsection{Local Hardy spaces on manifolds with bounded geometry}
We now move to the case of manifolds. The construction of a local Hardy space by Goldberg was extended to a certain class of Riemannian manifolds with strongly bounded geometry by
M.~Taylor \cite{T}, further generalised to the setting of metric measure spaces \cite{MVo}, and recently investigated in \cite{MaMV2}.  
In the case where $M$ is a complete Riemannian manifold with Ricci curvature bounded from below, the space $\ghu{M}$ can be characterised in several ways.  In particular, it admits an atomic and an ionic decomposition \cite{MVo}.  
If we assume further that $M$ has positive injectivity radius, then $\ghu{M}$ 
can be equivalently defined in terms of either the local heat maximal operator 
or the local Poisson maximal operator \cite{MaMV2}. 

For the reader's convenience, we report below one of the several equivalent definitions of $\ghu{M}$ given in \cite{MVo} in terms of atomic decompositions. As in \cite{MaMV2}, here we specialise to the case where the manifold $M$ belongs to the class $\cM_0$ of \emph{Riemannian manifolds with bounded geometry}, that is, complete connected Riemannian manifolds with positive injectivity radius and Ricci curvature bounded from below. We point out that the class $\cM_0$ is strictly larger than the class $\cM$ discussed in the Introduction, for $\cM_0$ contains also doubling manifolds, such as $\BR^n$ with the standard Euclidean structure, as well as compact manifolds. Here and in the sequel, we denote by $\mu$ the Riemannian measure on the manifold $M$.

\begin{definition}
Let $M$ be a manifold in the class $\cM_0$. A \emph{standard $\ghu{M}$-atom} is a function $a$ in $\lu{M}$ supported in a ball $B$ of radius at most one, satisfying the following conditions:
\begin{enumerate}[label=(\roman*)]
\item \emph{size condition}: $\|a\|_2 \leq \mu(B)^{-1/2}$;
\item \emph{cancellation condition}: $\int_B a \wrt\mu = 0$.
\end{enumerate}
A \emph{global $\ghu{M}$-atom} is a function $a$ in $\lu{M}$ supported in a ball $B$ of radius exactly one, satisfying the size condition above (but possibly not the cancellation condition). Standard and global $\ghu{M}$-atoms will be referred to simply as \emph{$\ghu{M}$-atoms}.
\end{definition}

\begin{definition}
Let $M$ be a manifold in the class $\cM_0$. The \emph{local atomic Hardy space} $\ghu{M}$ is the space of all functions $f$ in $\lu{M}$ that admit a decomposition of the form
\begin{equation}\label{f: atomicdec}
f = \sum_{j=1}^\infty \lambda_j a_j,
\end{equation}
where the $a_j$'s are $\ghu{M}$-atoms and $\sum_{j=1}^\infty |\lambda_j| < \infty$. The norm $\| f \|_{\frh^1}$ of $f$ is the infimum of $\sum_{j=1}^\infty |\lambda_j|$ over all decompositions \eqref{f: atomicdec} of $f$.
\end{definition}

We refer the reader to \cite{MVo} and \cite{MaMV2} for further details on $\ghu{M}$, including other equivalent characterisations.

Given a Riemannian manifold $M$ in the class $\cM_0$ and a positive number~$\tau$, 
we consider the translated Laplacian $\tau\cI+\cL$ and define
the associated \emph{translated Riesz transform} $\cR_\tau \defeq \nabla (\tau\cI + \cL)^{-1/2}$, 
and the \emph{local Riesz--Hardy} space
\begin{equation} \label{f: ghuBRn}
\ghuRtau{M}
\defeq \big\{f \in \lu{M}: \mod{\cR_\tau f} \in \lu{M}\big\}.
\end{equation}
We equip $\ghuRtau{M}$ with the norm
\[
\bignorm{f}{\ghuRtau{M}} \defeq \bignorm{f}{1} + \bignorm{\mod{\cR_\tau f}}{1}.
\]  

It is a classical result that $\ghuRtau{\BR^n}=\ghu{\BR^n}$ for all $\tau >0$. 
It was recently proved in \cite[Theorem~7.9]{MVe} that, if $M$ is in the class $\cM_0$ and $\tau$ is a large positive number, then 
$\ghuRtau{M}=\ghu{M}$ and their norms are equivalent. 
In fact, the requirement that $\tau$ be a large positive number is not needed for any manifold $M$ with bounded geometry,  
as shown in the next proposition that improves \cite[Theorem~7.9]{MVe}.  

\begin{proposition} \label{p: improvement tau}
Suppose that $M$ is Riemannian manifold in the class $\cM_0$,
 and let~$\tau$ be a positive number.  Then $\ghuRtau{M}=\ghu{M}$ and their norms are equivalent. 
\end{proposition}

\begin{proof}
Consider positive numbers $\tau$ and $\tau'$, and the function
\[
\vp_\tau^{\tau'} (z) = \left(\frac{\tau' + z}{\tau+z}\right)^{1/2}.
\] 
By Lemma \ref{l: Dunford class}, the function $\vp_\tau^{\tau'}$ belongs to the extended Dunford class $\cE(\bS_\theta)$ for any $\theta$ in $(\pi/2,\pi)$.
Since $\cL$ is the generator of a uniformly bounded semigroup of class $C_0$ on $\ghu{M}$ (see \cite[Theorem~3.1]{MaMV1}), the operator $\vp_\tau^{\tau'}(\cL)$ is bounded on $\ghu{M}$ by \cite[Theorem~2.3.3]{Haa} for every pair of positive numbers $\tau$ and $\tau'$.  The same also applies with $L^1(M)$ in place of $\ghu{M}$. As $\vp_{\tau'}^{\tau}(\cL)$ is the inverse of $\vp_\tau^{\tau'}(\cL)$, we conclude that $\vp_\tau^{\tau'}(\cL)$ is a topological vector space automorphism of $L^1(M)$ and also of $\ghu{M}$, for any $\tau,\tau' > 0$. Additionally,
\[
\cR_{\tau'} = \cR_\tau \, \vp_{\tau}^{\tau'}(\cL) ,
\]
so from \eqref{f: ghuBRn} we readily deduce that $\vp_{\tau}^{\tau'}(\cL) : \ghuRtaup{M} \to \ghuRtau{M}$ is bounded, and actually an isomorphism with inverse $\vp_{\tau'}^{\tau}(\cL)$.

Now, by \cite[Theorem~7.9]{MVe}, we can find $\tau'$ large enough that $\ghu{M} = \ghuRtaup{M}$. Then, for any positive number $\tau$,
\[
\ghuRtau{M} = \vp_{\tau}^{\tau'}(\cL) [\ghuRtaup{M}] = \vp_{\tau}^{\tau'}(\cL) [\ghu{M}] = \ghu{M},
\]
as required.
\end{proof}

\subsection{A family of Hardy type spaces}
We recall the definition of the Hardy type spaces $\gXga{M}$, $\ga>0$, recently introduced in \cite{MaMV1}, where $M$ is a manifold in the class $\cM$.
Consider the family $\big\{\cU_\tau \defeq \cL\,(\tau \cI+\cL)^{-1} : \tau>0\big\}$ of (spectrally defined) operators.
We also write $\cU$ instead of $\cU_1$ for brevity.
For any $\tau,\gamma>0$, the mapping $\cU_\tau^\ga$ is bounded on $\ghu{M}$ \cite[Section~3]{MaMV1}.
We emphasize that $\cU_\tau^{\ga}\big[\ghu{M}\big]$ is independent of $\tau$
and there exists a positive constant $C$ such that
\[
C^{-1} \bignorm{\cU^{-\ga} f}{\ghu{M}}
\leq  \bignorm{\cU_{\tau}^{-\ga} f}{\ghu{M}}
\leq C  \bignorm{\cU^{-\ga} f}{\ghu{M}}
\]
for every $f$ in $\cU^{\ga}\big[\ghu{M}\big]$ (see \cite[Proposition~3.5]{MaMV1}).

\begin{definition} \label{def: gXga}
We denote by $\gXga{M}$ the space $\cU^{\ga}\big[\ghu{M}\big]$, endowed with
the norm that makes $\cU^{\ga}$ an isometry, i.e.,
\[
	\bignorm{f}{\gXga{M}}
	\defeq \bignorm{\cU^{-\ga} f}{\ghu{M}}
	\quant f \in \cU^{\ga}\big[\ghu{M}\big].
\]
\end{definition}

See \cite{MaMV1} for further information on $\gXga{M}$, $\ga>0$.  
We record here a few useful properties of the spaces $\gXga{M}$ that will be important in the sequel.

\begin{proposition}\label{p: cinftyc}
Let $M$ be a manifold in the class $\cM$ and let $\gamma > 0$.
\begin{enumerate}[label=(\roman*)]
\item\label{en:cinftyc_decr} $\gXga{M}$ properly contains $\fX^{\gamma+\vep}(M)$ for all $\vep>0$.
\item\label{en:cinftyc_car} $f \in \gXga{M}$ if and only if $f \in \ghu{M}$ and $\cL^{-\gamma} f \in \ghu{M}$.
\item\label{en:cinftyc_cinftyc} If $\psi \in C^\infty_c(M)$, then $\cL^\ga \psi \in \gXga{M}$.
\end{enumerate}
\end{proposition}
\begin{proof}
Parts \ref{en:cinftyc_decr} and \ref{en:cinftyc_car} are proved in \cite[Proposition~4.14 and Proposition~4.2]{MaMV1}. As for part \ref{en:cinftyc_cinftyc}, note that
\[
\cU^{-\gamma} \cL^\gamma \psi = (1+\cL)^{\gamma} \psi = (1+\cL)^{\gamma-N} (1+\cL)^N \psi
\]
for all $N \in \BN$.  Since $\psi \in C^\infty_c(M)$, we deduce that $(1+\cL)^N \psi \in C^\infty_c(M)$, 
and therefore $(1+\cL)^N \psi \in \ghu{M}$; the conclusion follows by the 
$\frh^1$ boundedness of $(1+\cL)^{\gamma-N}$ for $N > \gamma$ \cite[Theorem~3.1]{MaMV1}.
\end{proof}

\subsection{Negative powers of the Laplacian and Riesz transform}
As a consequence of the spectral gap assumption for manifolds $M$ of the class $\cM$, the powers $\cL^{-\ga}$, $\ga>0$, are bounded on $L^p(M)$ for any $p \in (1,\infty)$, though this fails for $p=1$. Nevertheless we can show that $\cL^{-\ga}$ maps $L^1(M)$ into $L^1_{\loc}(M)$.

\begin{proposition}\label{p: apriori}
Let $M$ be a manifold in the class $\cM$ and $\ga \in (0,\infty)$. Then, there exists a positive constant $C$ such that, for all $f \in \lu{M}$ and $o \in M$,
\[
\| \cL^{-\ga} f \|_{L^1(B_{R}(o))} \leq C R^\ga \| f \|_{L^1} \quant R \geq 1.
\]
\end{proposition}
\begin{proof}
Recall that $\cH_t = \,e^{-t\cL}$ denotes the heat semigroup. By functional calculus, we can write, at least formally,
\[
\cL^{-\ga} = c_\ga \int_0^\infty t^\ga \cH_t \dtt{t} ,
\]
where $c_\ga := 1/\Gamma(\ga)$. As a consequence, for all $T > 0$,
\[
\cL^{-\ga} f = c_\ga \int_0^T t^\ga \cH_t f \dtt{t} + c_\ga \int_T^\infty t^\ga \cH_t f \dtt{t} =: f_T^0 + f_T^\infty.
\]
Now, as $\cH_t$ is a contractive semigroup on $\lu{M}$,
\[
\| f^0_T \|_{L^1} \leq c_\gamma \|f\|_{L^1} \int_{0}^T t^\ga \dtt{t} \leq C T^\ga \|f\|_{L^1},
\]
since $\ga > 0$. On the other hand, in light of the ultracontractive estimate
\[
\|\cH_t\|_{1 \to 2} \leq C \e^{-bt} \quant t\geq 1,
\]
where $b>0$ is the bottom of the $L^2$-spectrum of $\cL$ (see, e.g., \cite[eq.\ (2.4)]{MaMV1}), we deduce that
\[
\| f^\infty_T \|_{L^2} \leq C \|f\|_{L^1} \int_{T}^\infty t^{\ga-1} \e^{-bt} \wrt t \leq C T^{\ga-1} \e^{-bT} \|f\|_{L^1}
\]
for all $T \geq 1$. As $M$ has exponential volume growth (see, e.g., \cite[eq.\ (2.3)]{MaMV2}), there exists $\beta > 0$ such that
\[
\mu(B_R(o)) \leq C \e^{2\beta R} \quant R \geq 1
\]
and consequently, by the Cauchy--Schwarz inequality,
\[
\| f^{\infty}_T \|_{L^1(B_R(o))} \leq C \e^{\beta R} \| f^{\infty}_T \|_{L^2} \leq C T^{\ga-1} \e^{\beta R - bT}.
\]
In conclusion,
\[
\|\cL^{-\ga} f \|_{L^1(B_R(o))} \leq C T^{\ga} \left[ 1 + T^{-1} \e^{\beta R - bT} \right]
\]
for all $R ,T \geq 1$, and the desired estimate follows by taking $T = (1+\beta/b) R$.
\end{proof}

On the basis of the previous estimate, the Riesz transform $\cR f = \nabla \cL^{-1/2} f$ of any function $f \in \lu{M}$ is well-defined at least in the sense of distributions.
We can then define the Riesz--Hardy space $\huR{M}$ as described in \eqref{f: rieszhardy} in the Introduction, and equip it with the norm 
$\norm{f}{\lu{M}}+\bignorm{\mod{\cR f}}{\lu{M}}$.

\subsection{The Federer--Fleming inequality}
As we shall see, $\huR{M}$ can be characterised as one of the spaces $\gXga{M}$ introduced above. An important tool in this characterisation is the validity of the following $L^1$ Sobolev inequality, also known as the Federer--Fleming inequality \cite{Bu,Ch}, which is again a consequence of the spectral gap assumption.

\begin{proposition}\label{p: FF}
Let $M$ be a manifold in the class $\cM$. There exists a positive constant $\kappa$ such that, if $h \in L^1_\loc(M)$ satisfies the estimate
\begin{equation}\label{f: apriori}
\| h \|_{L^1(B_R(o))} \leq C R    \quant R \geq 1
\end{equation}
for some $o \in M$ and $C > 0$, and moreover $|\nabla h| \in \lu{M}$, then $h \in \lu{M}$ and
\begin{equation}\label{f: FF}
\| h \|_{L^1} \leq \kappa \, \| |\nabla h| \|_{L^1}.
\end{equation}
\end{proposition}
\begin{proof}
As the manifold $M$ is in the class $\cM$, the Cheeger isoperimetric constant of $M$ is strictly positive (see \cite{Bu} and \cite[Theorem 9.5]{CMM1}), whence the Federer--Fleming inequality \eqref{f: FF} is known to hold for all $h \in C^\infty_c(M)$ (see \cite[Theorem V.2.1, p.\ 131]{Ch}, and also \cite[eq.\ (2.6)]{MaMV1}). On the other hand, 
as $M$ is complete,
 the space $C^\infty_c(M)$ is dense in the Sobolev space
\[
W^{1,1}(M) = \{ f \in \lu{M} : |\nabla f| \in \lu{M} \}
\]
(see, e.g., \cite[Theorem 2.7]{He} or \cite[Proposition 5.5]{CMa}), thus the inequality \eqref{f: FF} holds for all $h \in W^{1,1}(M)$. So it just remains to show that, if $h \in L^1_\loc(M)$ is only assumed to satisfy the a-priori estimate \eqref{f: apriori} and $|\nabla h| \in \lu{M}$, then actually $h \in \lu{M}$.

To this purpose, let us consider, for all $R \geq 1$, a compactly supported Lipschitz cutoff $\psi_R : M \to [0,1]$ which is identically $1$ on $B_R(o)$, vanishes outside $B_{2R}(o)$, and satisfies
\begin{equation}\label{f: cutoff}
\| |\nabla \psi_R| \|_\infty \leq C R^{-1} ;
\end{equation}
the existence of such cutoffs is a consequence of the completeness of $M$ (see, e.g., \cite[Lemma 2.2]{strichartz}). 

Let now $h$ satisfy the assumptions of Proposition \ref{p: FF}. As $h \in L^1_{\loc}(M)$, we deduce that $\psi_R h$ is  in $L^1(M)$. Moreover, as $|\nabla h| \in L^1(M)$, by the Leibniz rule,
\[
|\nabla (\psi_R h)| \leq  \psi_R |\nabla h| + |\nabla \psi_R| |h|,
\]
so $|\nabla (\psi_R h)|$ is in $\lu{M}$ too, with
\begin{equation}\label{f: leibniz_est}
\begin{split}
\||\nabla (\psi_R h)|\|_1 &\leq  \||\nabla h|\|_1 + C R^{-1} \|h\|_{L^1(B_{2R}(o))} \\
&\leq \||\nabla h|\|_1 + C ;
\end{split}
\end{equation}
in the first inequality we used \eqref{f: cutoff} and information on the support of $\psi_R$, while the second inequality follows from \eqref{f: apriori}.

As $\psi_R h \in W^{1,1}(M)$, we can then apply \eqref{f: FF} to $\psi_R h$ and obtain that
\[
\| \psi_R h \|_1 \leq \kappa \| |\nabla(\psi_R h)| \|_1 \leq \kappa \| |\nabla h| \|_1 + C
\]
for all $R \geq 1$. If we take the limit as $R \to +\infty$, then by monotone convergence we deduce that $h \in \lu{M}$, as required.
\end{proof}

\subsection{Characterisation of the Riesz--Hardy space}
In view of Proposition~\ref{p: hu ghu Rn}, the following characterisation of $\huR{M}$ highlights 
an interesting noneuclidean phenomenon.

\begin{theorem} \label{t: char Riesz}
Let $M$ be a manifold in the class $\cM$.  Then $\huR{M}=\gXum{M}$, with equivalence of norms.
\end{theorem}

This characterisation substantially refines the result in \cite[Theorem~5.3]{MaMV1}, where $\gXum{M}$ is identified with the space $\{ f \in \frh^1(M) : |\cR f| \in L^1(M) \}$. While a number of ideas from the proof of \cite[Theorem~5.3]{MaMV1} also apply here, for the reader's convenience we opt to provide a self-contained proof of Theorem \ref{t: char Riesz} below. Beside favouring the readability of the argument, this also gives us the opportunity to clarify why the Federer--Fleming inequality is applicable in this context, a point that was somewhat skimmed over in the proof presented in \cite{MaMV1}.

\begin{proof}
The proof of the inclusion $\gXum{M} \subseteq \huR{M}$ is analogous to that of the inclusion $\cU^{1/2}[\frh^1(\BR^n)] \subseteq \huR{\BR^n}$ given in Proposition \ref{p: hu ghu Rn}. Indeed, if $f$ is in $\gXum{M}$, then $f$ is in $\frh^1(M)$, thus in $\lu{M}$, and moreover $f = \cU^{1/2} g$ for some $g \in \frh^1(M)$. In addition,
\begin{equation}\label{f: riesz_local}
\cR f
= \nabla \cL^{-1/2} f
	= \nabla (\cI+\cL)^{-1/2} \cU^{-1/2} f
= \cR_1 g.
\end{equation}
So from Proposition \ref{p: improvement tau} and the fact that $g \in \frh^1(M)$ we deduce that $\cR f \in \lu{M}$, and in conclusion $f \in \huR{M}$.

As for the opposite inclusion, if $f$ is in $\huR{M}$, then $f$ and $\bigmod{\cR f}$ are in $\lu{M}$.
If $h := \cL^{-1/2} f$, then, from Proposition \ref{p: apriori}, we know that $h \in L^1_\loc(M)$ and $\|h\|_{L^1(B_R(o))} \leq C R^{1/2}$ for all $R \geq 1$. Moreover $\nabla h = \cR f$, thus $|\nabla h| \in \lu{M}$. We can then apply Proposition \ref{p: FF} to $h$ and deduce that $h = \cL^{-1/2} f$ is in $\lu{M}$.

Notice now that
\begin{equation}\label{f: radiceinversa}
\cU^{-1/2} =  (\cI + \cL)^{1/2} \cL^{-1/2}  = \vp(\cL) (\cI + \cL^{-1/2}),
\end{equation}
where 
\[
\vp(z) := \frac{(1+z)^{1/2}}{1+z^{1/2}}.
\]
It is easily seen that that $\vp$ is in $\cE(\bS_\theta)$ for all $\theta \in (\pi/2,\pi)$: indeed, by parts \ref{en:dunford_algebra} and \ref{en:dunford_examples} of Lemma \ref{l: Dunford class}, the reciprocal function
\[
\frac{1}{\vp(z)} = \frac{1}{(1+z)^{1/2}} + \frac{z^{1/2}}{(1+z)^{1/2}}
\]
is in $\cE(\bS_\theta)$, thus $\vp$ is too by part \ref{en:dunford_reciprocal} of the same Lemma. So $\vp(\cL)$ is bounded on $\lu{M}$ by \cite[Theorem 2.3.3]{Haa}. From \eqref{f: radiceinversa} we then deduce that the function
\[
g := \cU^{-1/2} f = \vp(\cL) (f + \cL^{-1/2} f)
\]
is in $\lu{M}$, as $f$ and $\cL^{-1/2} f$ are.
Moreover, as in \eqref{f: riesz_local}, $\cR f= \cR_1 g$. 
Therefore both $g$ and $|\cR_1 g|$ are in $L^1(M)$, that is, $g \in \frh^1_{\cR_1}(M)$. However, by Proposition \ref{p: improvement tau}, this is equivalent to saying that $g \in \ghu{M}$, and consequently $f = \cU^{1/2} g \in \gXum{M}$, by definition.
\end{proof}

As a consequence of the above characterisation, in stark contrast to the Euclidean case, we can rule out that $\huR{M}$ can be characterised in terms of atomic decompositions, at least for a  subclass of the manifolds of class $\cM$.

\begin{corollary} \label{c: char Riesz}
Suppose that $M$ is a rank one symmetric space of the noncompact type.   Then
the Riesz--Hardy space $\huR{M}$ does not admit an atomic characterisation. More precisely, the set of the compactly supported elements of $\huR{M}$ is not dense in $\huR{M}$.
\end{corollary}

\begin{proof}
According to \cite[Theorem~4.15]{MaMV1}, the space $\gXum{M}$ does not admit an atomic decomposition.
Since, by Theorem~\ref{t: char Riesz}, $\huR{M}$ agrees with $\gXum{M}$, the conclusion follows.
\end{proof}

\section{Maximal operators and function spaces for the class \texorpdfstring{$\cM$}{M}}
\label{s: Maximal}

In this section we consider various maximal operators and function spaces associated thereto on any manifold $M$ in the class $\cM$.
We collect here their definitions and some of their elementary properties.

\subsection{Maximal operators and associated Hardy type spaces}
Recall that $\cH_t=\e^{-t\cL}$ and that $\cP_t=\e^{-t\cL^{1/2}}$: thus, 
$\{\cH_t\}$ and $\{\cP_t\}$ denote the heat and the Poisson semigroups on $M$, respectively.  
For every $c\in \BR$ we denote by $\cH^c_*$ the \emph{weighted heat maximal operator}, defined by 
\[
\cH^c_* f 
= \sup_{t>0} \, \langle t\rangle^c \, \bigmod{\cH_t f}\,,
\]
where $\langle t\rangle =\max\{1,t\}$. 
When $c=0$ we simply denote by $\cH_*$ the maximal operator $\cH_*^0$. We denote by $H^1_{\cH,c}$ the space
\[
H^1_{\cH,c}=\{f\in L^1(M):  \cH^c_* f \in L^1(M)\}\,,
\]
endowed with the norm 
\[
\|f\|_{H^1_{\cH,c}}=\|f\|_{L^1}+ \|\cH^c_* f \|_{L^1}\,.
\]
Clearly $\{H^1_{\cH,c}(M): c\in \BR\}$ is a decreasing family of spaces.     
We denote by $\maxHloc$ 
the \emph{local heat maximal operator}, defined by 
\[
\maxHloc f 
= \sup_{0<t\leq 1} \bigmod{\cH_t f}
\,.
\]
We define the \emph{local heat Hardy space} as 
\[
\frh^1_{\cH}(M)=\{f\in L^1(M) : \maxHloc f \in L^1(M)\}\,,
\]
endowed with the norm 
\[
\|f\|_{\frh^1_{\cH}}=\|f\|_{L^1}+\|\maxHloc f\|_{L^1}\,.
\]
Much in the same vein,  for any real number $c$ in $\BR$ we consider the \emph{Poisson maximal operator $\maxPc$ with parameter $c$}, which acts on 
a function $f$ on $M$ by
\[
\maxPc f 
= \sup_{t>0} \,  \langle t\rangle^c \, \bigmod{\cP_tf}.
\]
We shall write $\maxP$ instead of $\maxP^0$.  We then define $\huPc{M}$ by 
\[
\huPc{M}
= \{f \in \lu{M}: \maxPc f \in \lu{M} \}.  
\]
We equip $\huPc{S}$ with the norm 
\[
\bignorm{f}{H^1_{\cP,c} }
= \bignorm{f}{L^1} + \bignorm{\maxPc f}{L^1}. 
\]
We denote by $\maxPloc$ 
 the \emph{local Poisson maximal operator}, defined by 
\[
\maxPloc f 
= \sup_{0<t\leq 1} \bigmod{\cP_t f}
\,.
\]
We define the \emph{local Poisson Hardy space} as 
\[
\frh^1_{\cP}(M)=\{f\in L^1(M) : \maxPloc f\in L^1(M)\}\,,
\]
endowed with the norm 
\[
\|f\|_{\frh^1_{\cP}}=\|f\|_{L^1}+\|\maxPloc f\|_{L^1}\,.
\]
We shall need the following result, which relates $\frh^1_{\cH}(M)$ and $\frh^1_{\cP}(M)$ with the Goldberg type
space $\ghu{M}$, and is valid more generally for any manifold $M$ in the class $\cM_0$.

\begin{proposition}\label{p: H1Hloc}
The spaces $\frh^1(M)$, $\frh^1_{\cH}(M)$ and $\frh^1_{\cP}(M)$ 
coincide and their norms are mutually equivalent. 
\end{proposition}

\begin{proof}
See \cite[Corollary~5.4]{MaMV2}.
\end{proof}

\subsection{Relationships between \texorpdfstring{$H^1_{\cH,c}(M)$}{H1Hc(M)}, \texorpdfstring{$H^1_{\cP,c}(M)$}{H1Pc(M)} and \texorpdfstring{$\gXga{M}$}{Xgamma(M)}}  
In this subsection we establish some inclusions between $H^1_{\cH,c}(M)$, $H^1_{\cP,c}(M)$ and $\gXga{M}$, whenever $M$
belongs to the class $\cM$ defined in the Introduction.

It is well known that the the Poisson semigroup $\cP_t$ can be subordinated to the heat semigroup $\cH_t$ via the formula
\begin{equation} \label{f: sub Poisson}
\cP_t = t \ioty \frac{\e^{-t^2/(4s)}}{\sqrt{4\pi s}} \, \cH_s \dtt s
\end{equation} 
(see, for instance, \cite[formula (2), p.~260]{Y} or \cite[formula (*), p.~47]{St1}).

One of the results we want to prove in this generality is that $\gXum{M} \subseteq  \huPc{M}$ for all $c<1$.  This will
generalise a result of Anker~\cite{A1}, who proved, in the case where $M$ is a Riemannian symmetric space of the noncompact type, that for every $c <1$ there exists a constant $C$ such that 
\begin{equation} \label{f: Anker maximal}
\bignorm{\maxPc f}{L^1}
\leq C \, \big(\bignorm{f}{L^1} + \bignorm{\mod{\cR f}}{L^1} \big),
\end{equation}
where $\cR$ denotes the Riemannian Riesz transform.  

From Proposition \ref{p: cinftyc}~\ref{en:cinftyc_car} we know that a function $f$ belongs to $\gXga{M}$ if and only if both $f$ and $\cL^{-\ga}f$
belong to $\ghu{M}$.  In order to prove the inclusion $\gXum{M} \subseteq  \huPc{M}$ for $c<1$, and more generally that 
$\gXga{M} \subseteq \huPc{M}$ for $c <2\ga$ it is natural to write $\cP_t f$ as $\cL^{\ga} \cP_t \cL^{-\ga}f$, and
to look for estimates for the operator $\cL^{\ga} \cP_t$.  
We can obtain subordination formulae for the operators $\cL^{\gamma} \cP_t$ in terms of the heat semigroup similar to 
formula \eqref{f: sub Poisson} for $\cP_t$, as shown by taking $\alpha=1/2$ in the following technical lemma.  

\begin{lemma} \label{l: est Lpt}
For $\al \in (0,1)$ and $\ga \geq 0$, let $F_{\alpha,\gamma} : (0,\infty) \to \BC$ be defined by contour integration as follows:
\begin{equation} \label{f: Fgamma def}
F_{\al,\ga}(s) = \frac{1}{2\pi i} \int_{\sigma-i\infty}^{\sigma+i\infty} s \, \e^{sz-z^{\al}} z^{\ga} \wrt z
\end{equation}
for any $\sigma > 0$. 
	\begin{enumerate}[label=(\roman*)]
		\item\label{en:est_Lpt_pointwise}
			There exist constants $C,\rho>0$ (possibly depending on $\alpha$ and $\gamma)$ such that
			\begin{equation} \label{f: Fgamma est}
			|F_{\al,\ga}(s)| \leq C s^{-\ga} \, \e^{-2\rho s^{-\al/(1-\al)}}.
			\end{equation}
			Moreover
			\begin{equation} \label{f: laplace tr}
			z^{\ga} \e^{-z^{\al}} = \ioty F_{\al,\ga}(s) \, \e^{-sz} \dtt s
			\end{equation}
			for every complex number $z$ with $\Re z > 0$.
		\item\label{en:est_Lpt_max}
			For all real numbers $c<\ga/\al$,
			\[
			\ioty \sup_{t\geq 1} \, t^c 
          		\left|  t^{-\ga/\al} F_{\al,\ga}(s/t^{1/\al}) \right|  \dtt s < \infty.
			\]
	\end{enumerate}
\end{lemma}

We point out that, in the case $\gamma=0$, a more precise estimate of $F_{\al,\ga}(s)$ for $s \to \infty$ can be proved than the one in \eqref{f: Fgamma est}, as shown in \cite[Lemma 5.2]{MaMV2}. However we will not need this here.

\begin{proof}
We first prove part \ref{en:est_Lpt_pointwise}.
The construction of the function $F_{\alpha,\gamma}$ satisfying \eqref{f: laplace tr} through the contour integration 
\eqref{f: Fgamma def} is a small modification of that in \cite[Section IX.11]{Y}, and can be justified through the theory 
of the Laplace transform.

Let $a := \cos(\al \pi/2)/3$. Since $\Re(z^\al) \geq 3a |z|^\al$ whenever $\Re z > 0$, from 
the representation \eqref{f: Fgamma def} we obtain that, for any $\sigma \geq (a\alpha)^{1/(1-\alpha)}$,
\[\begin{split}
|F_{\al,\ga}(s)| 
&\leq \frac{1}{2\pi} s \e^{s\sigma} \sigma^{\ga+1} \ir (1+t^2)^{\ga/2} \e^{-3a \sigma^\al (1+t^2)^{\al/2}} \wrt t \\
&\leq C  s \e^{2s\sigma} \e^{-2a\sigma^\al },
\end{split}\]
where the constant $C$ may depend on $\al$ and $\ga$ but not on $\sigma$;
by choosing $\sigma = (a\al/s)^{1/(1-\al)}$ and observing that
\[
s (a\al/s)^{1/(1-\al)}- a(a\al/s)^{\al/(1-\al)} = -(1-\al)  \al^{\al/(1-\al)}  a^{1/(1-\al)} s^{-\al/(1-\al)},
\]
we obtain the estimate in \eqref{f: Fgamma est} in the case $s \leq 1$. As for the remaining case, by changing the 
contour of integration in \eqref{f: Fgamma def} to the concatenation of the two half-lines 
$z = r \e^{-i\theta}$ ($-\infty < -r < 0$) and $z = r \e^{i\theta}$ ($0 < r < \infty$) for any $\theta \in (\pi/2,\pi)$, we obtain, 
much as in \cite[p.~263]{Y}, the representation
\[
F_{\al,\ga}(s) 
= \frac{1}{\pi} \ioty s r^{\ga} \e^{sr \cos\theta 
	- r^\al \cos(\al\theta)} \sin\left(sr\sin\theta-r^\al \sin(\al\theta) + \theta(1+\ga)\right) \wrt r;
\]
if we take $\theta = \pi/(1+\alpha)$, then $\alpha\theta = \pi-\theta$, thus $k := \cos(\alpha\theta) = -\cos(\theta) > 0$ and
\[
|F_{\al,\ga}(s)| 
\leq C \ioty s r^{\ga} \e^{-k(sr+r^\al)} \wrt r 
\leq C s^{-\ga}
\]
for $s \geq 1$. This completes the proof of part \ref{en:est_Lpt_pointwise}.

Next we prove part \ref{en:est_Lpt_max}. 
Note that, by \eqref{f: Fgamma est},
\[\begin{split}
\sup_{t\geq 1} \,  t^c \left|  t^{-\ga/\al} F_{\al,\ga}(s/t^{1/\al}) \right| 
&\leq Cs^{\al c-\ga} \sup_{t \geq 1} \, (t/s^{\al})^c \e^{-2\rho (t/s^\al)^{1/(1-\al)}} \\
&\leq Cs^{\al c-\ga} \e^{-\rho s^{-\al/(1-\al)}},
\end{split}\]
and the latter is integrable with respect to $\mathrm{d}s/s$ provided $c < \ga/\al$.
\end{proof}

In view of the equality $\gXum{M} = H_{\cR}^1(M)$ proved in Section~\ref{s: Hardy--Riesz space}, 
the estimate \eqref{f: Anker maximal} corresponds to the case $\ga=1/2$ of Theorem \ref{t: maximal heat manifold}~\ref{en:maxhpman_poisson_reverse} below.

\begin{theorem}\label{t: maximal heat manifold}
Suppose that $\ga$ is a positive number.  The following hold:
\begin{enumerate}[label=(\roman*)]
\item\label{en:maxhpman_heat}
$\huHc{M}\subseteq\ghu{M}$ for all $c \in \BR$ and $\huHc{M} \subseteq \gXga{M}$ for all $c>\ga$;
\item\label{en:maxhpman_poisson}
$\huPc{M}\subseteq \ghu{M}$ for all $c \in \BR$ and
$\huPc{M}\subseteq \gXga{M}$ for all $c>2\gamma$;
\item\label{en:maxhpman_poisson_reverse}
$\huPc{M} = \ghu{M}$ for all $c<0$ and 
$\gXga{M} \subseteq \huPc{M}$ for all $c<2\gamma$;
\item\label{en:maxhpman_heatpoisson}
$\huHc{M}\subseteq H^1_{\cP,2c}(M)$ for all $c \in \BR$.
\end{enumerate}
\end{theorem}

\begin{proof}
First we prove part \ref{en:maxhpman_heat}.  Suppose that $f$ is in $\huHc{M}$, i.e. $\cH^c_* f$ is in $\lu{M}$.  Then
$\maxHloc f$ is in $\lu{M}$, whence $f$ is in $\ghu{M}$, by Proposition~\ref{p: H1Hloc}, and
\[
\|f\|_{\frh^1} \leq C \, \| \maxHc f \|_{L^1}.
\]

Now suppose that $c>\ga$ and assume that $f$ is in $\huHc{M}$.  We have already proved that $f$ is in $\ghu{M}$.
By Proposition \ref{p: cinftyc}~\ref{en:cinftyc_car}, in order to conclude~$f$ is in $\gXga{M}$, it suffices to prove that 	
$g\defeq\cL^{-\gamma}f$ is in $\ghu{M}$, and there exists a constant $C$ such that 
\[
	\bignorm{\maxHloc (\cL^{-\ga}f)}{L^1}
	\leq C \bignorm{\maxHc f}{L^1} 
	\quant f \in \huHc{M}.  
\]
By spectral theory,
\begin{equation} \label{f: Calderon heat}
	\cI = c_\gamma \ioty (t\cL)^\gamma \, \e^{-t\cL}  \, {\dtt t},
\end{equation} 
where $c_\gamma \defeq 1/\Ga(\ga)$.  Thus, for each $s$ in $(0,1]$,
\[
\cH_s g 
	= c_\gamma \ioty \, t^{\gamma} \, \cH_{t+s} f \, {\dtt t}  
	= c_\gamma \ioty \, \frac{t^{\gamma}}{\langle t+s\rangle^c}\, \langle t+s\rangle^c\, \cH_{t+s} f  {\dtt t}.  
\]
By taking the supremum with respect to $s$ in $(0,1]$, we see that
\[
\bignorm{\maxHloc g}{L^1} 
\leq C \bignorm{\maxHc f}{L^1} \ioty t^{\ga-1} \langle t\rangle^{-c} \wrt t.
\]
The integral on the right hand side is convergent, because $c>\gamma>0$, and therefore $g$ is in $\ghu{M}$, by Proposition \ref{p: H1Hloc}, as required.

The proof of part \ref{en:maxhpman_poisson}
is entirely analogous to that of part \ref{en:maxhpman_heat}.
One simply needs to replace $\gamma$ with $2\gamma$, $\cH$ with $\cP$, as well as $\cL$ with $\cL^{1/2}$ in the reproducing formula \eqref{f: Calderon heat}.  We omit the details. 

Next, we prove part \ref{en:maxhpman_poisson_reverse}.
Due to the inclusion already established in part \ref{en:maxhpman_poisson}, we only need to prove that
$\ghu{M} \subseteq \huPc{M}$ for all $c < 0$, and
$\gXga{M} \subseteq \huPc{M}$ for all $\gamma > 0$ and $c<2\gamma$. In light of the characterisation of $\gXga{M}$ in Proposition \ref{p: cinftyc}~\ref{en:cinftyc_car}, we are effectively reduced to proving that, for all $\gamma \geq 0$ and $c < 2\gamma$,
\[
\text{if } f \in \ghu{M} \text{ and } \cL^{-\gamma} f \in \ghu{M}, \text{ then } f \in \huPc{M}.
\]

Let then $\ga \geq 0$, and assume that both $f$ and $\cL^{-\ga} f$ are in $\ghu{M}$.  As $f$ is in $\ghu{M}$, we deduce that $\maxPloc f$ is in $\lu{M}$,
by Proposition~\ref{p: H1Hloc}, and 
\begin{equation} \label{f: incl I}
	\bignorm{\maxPloc f}{L^1} \leq C \bignorm{f}{\frh^1}.  
\end{equation}
Let $c< 2\ga$; in order to prove that $f$ is in $\huPc{M}$, it remains to show that 
$\sup_{t>1}\, t^c \mod{\cP_t f}$ is in $\lu{M}$.  We write $\cP_t f = \cL^\ga \cP_t\cL^{-\ga}f$, and set $g\defeq\cL^{-\ga}f$. 
Formula \eqref{f: laplace tr}, applied with $\al=1/2$, and spectral theory yield the subordination formula
\begin{equation} \label{f: Lpt sub}
\cL^{\gamma} \cP_t g = t^{-2\ga}  \ioty F_{1/2,\ga}(s) \cH_{st^2}g \dtt s 
=  \ioty \left[ t^{-2\ga} F_{1/2,\ga}(s/t^2) \right] \cH_s g \dtt s. 
\end{equation}
Then 
\[
\left\|\sup_{t>1}\, t^c \mod{\cP_t f}\right\|_{L^1}
\leq \ioty \sup_{t>1} t^c \bigmod{ t^{-2\ga} F_{1/2,\ga}(s/t^2)} \bignorm{\cH_s g}{L^1} \dtt s
\leq C \bignorm{g}{L^1};
\]
the last inequality follows from the contractivity of $\cH_s$ on $\lu{M}$ and Lemma~\ref{l: est Lpt}~\ref{en:est_Lpt_max} (which applies
for $c<2\ga$). Combining this inequality with \eqref{f: incl I} finally yields
\[
\bignorm{\cP_*^c f}{L^1} \leq C \left[ \|f\|_{\frh^1} + \|\cL^{-\ga} f\|_{\frh^1}  \right],
\]
as required.

Finally, we prove part \ref{en:maxhpman_heatpoisson}. If $c<0$, then $H^1_{\cH,c}(M) \subseteq \ghu{M} = H^1_{\cP,2c}(M)$ by parts \ref{en:maxhpman_heat} and 
\ref{en:maxhpman_poisson_reverse}. Assume instead that $c\geq 0$.
By changing variables in the subordination formula  \eqref{f: sub Poisson}, we see that 
\[
\cP_t = \frac{1}{\sqrt \pi} \ioty v^{-1/2} \, \e^{-v} \, \cH_{t^2/4v} \wrt v,
\]
and consequently, since $c\geq 0$,
\[
t^{2c} |\cP_t f| \leq \frac{4^{c}}{\sqrt \pi} \ioty v^{c-1/2} \, \e^{-v} \, \langle t^2/4v \rangle^{c} |\cH_{t^2/4v} f| \wrt v.
\]
Since $\ioty v^{c-1/2} \, \e^{-v} \wrt v < \infty$, this implies the pointwise estimate
\[
\sup_{t > 0} t^{2c} |\cP_t f| \leq C \cH_*^c f,
\]
and therefore
\[
\cP_*^{2c} f \leq \sup_{0 < t \leq 1} |\cP_t f| + \sup_{t\geq 1} t^{2c} |\cP_t f| \leq C (\cH_* f + \cH_*^c f) \leq C \cH_*^c f,
\]
as desired.
\end{proof}

From Theorem~\ref{t: maximal heat manifold} we deduce the following strict inclusions, which imply among other things that the Riesz Hardy space and the Poisson--Hardy space on $M$ cannot coincide.

\begin{corollary} \label{c: maximal heat manifold}
For each $\ga,\vep>0$ the following proper inclusions hold:
\[
H_{\cP,2\ga+\vep}^1(M) \subset \gXga{M} \subset H_{\cP,2\ga-\vep}^1(M).
\]
In particular,
\[
H_{\cP,1+\vep}^1(M) \subset \huR{M} \subset H_{\cP,1-\vep}^1(M)
\]
and therefore
\[
\huR{M} \subset \huP{M}.
\]
\end{corollary}

\begin{proof}
Up to the fact that the inclusions are strict, the first chain of inclusions follows by combining parts \ref{en:maxhpman_poisson} and \ref{en:maxhpman_poisson_reverse} of Theorem~\ref{t: maximal heat manifold}. 
Moreover, in light of the equality $\gXum{M} = \huR{M}$ proved in Theorem~\ref{t: char Riesz}, the second chain of inclusions is a restatement of the first one in the case where $\ga = 1/2$, and the last inclusion follows by further specialising to $\vep = 1$. So we only need to check the strictness of the first chain of inclusions.

To prove that $H_{\cP,2\ga+\vep}^1(M)$ is properly contained in $\huR{M}$, we argue as follows.  Choose~$\de$ so that $2\ga+\vep >\de>2\ga$.   
Note that $H_{\cP,2\ga+\vep}^1(M) \subseteq \fX^{\de/2}(M)$, by Theorem~\ref{t: maximal heat manifold}~\ref{en:maxhpman_poisson}.
Since the family of spaces $\{\gXga{M}: \ga>0\}$ is strictly decreasing (see Proposition \ref{p: cinftyc}~\ref{en:cinftyc_decr}) and $\de>2\ga$,
$\gXga{M}$ contains $\fX^{\de/2}(M)$ properly, whence $\gXga{M}$ contains $H_{\cP,2\ga+\vep}^1(M)$ properly, as required.

The proof that $\gXga{M}$ is properly contained in $H_{\cP,2\ga-\vep}^1(M)$ is similar, and is omitted. 
\end{proof}

In light of Corollary~\ref{c: maximal heat manifold}, a natural question is whether $\huPu{M}$ agrees with $\huR{M}$.  
We shall see in Section~\ref{s: Maximal heat operator} that the answer is in the negative when $M$ is a Damek--Ricci space.  
In fact, a similar remark applies to $H_{\cP,2\gamma}^1(M)$ and $\gXga{M}$ for any $\gamma>0$, as shown in Theorem~\ref{t: comparison weighted poisson} below.

\section{Noninclusions on Damek--Ricci spaces}
\label{s: Maximal heat operator}

The results in this section state that certain inclusions between various type of Hardy spaces introduced in the previous sections
fail in the case where the Riemannian manifold $M$ is a Damek--Ricci space, which we shall henceforth denote by $S$.  
The main reason for restricting to Damek--Ricci spaces is that we need rather precise estimates for the heat semigroup that 
we obtain via spherical Fourier analysis.
The noninclusions we prove on Damek--Ricci spaces are likely to hold on all symmetric spaces of the noncompact type.  
However, this would add little knowledge at the expense of a presumably long analysis based on quite intricated and technical proofs.  
Altogether, we believe this is not worth doing here.  

\subsection{Damek--Ricci spaces}

We refer to \cite{ADY, As, DR, DiB} for more details on the analysis on Damek--Ricci spaces. We just recall here that a  
Damek--Ricci space $S$ is the one-dimensional harmonic extension of an Heisenberg-type group $N$ obtained 
by making $\BR^+$ act on $N$ by automorphic dilations.
In particular, $S$ is a connected Lie group, with a Riemannian structure which is invariant by left translations, and the Riemannian measure $\mu$ is a left Haar measure.
We denote by $Q$ the homogeneous dimension 
of $N$ and by $n$ the dimension of $S$. It is known that in this case the bottom of the 
$L^2$-spectrum of the Laplace--Beltrami operator $\cL$ is $b=Q^2/4$.

A radial function on $S$ is a function that depends only on the distance from the base point $o$. If $f$ is radial, then we may write $f(r)$ in place of $f(x)$ when $d(x,o)=r$; moreover
\begin{equation}\label{f: misuraDR}
\int_S f(x) \wrt \mu(x)= \ioty f({r}) A({r}) \wrt r\,, \quad\text{where }
A(r) \asymp \begin{cases}r^{n-1}&r<1\\
\e^{Qr}&r\geq 1\,.
\end{cases}
\end{equation}

A radial function is called \emph{spherical} if it is an eigenfunction of $\cL$ and if it takes value $1$ at $o$. 
For $\la\in\BC$ we denote by $\phi_{\la}$ the spherical function such that
\begin{equation}\label{f: sph_eigen}
\cL \phi_{\la}=(\la^2+Q^2/4)\, \phi_{\la}.
\end{equation}
The \emph{spherical transform} of an integrable radial function is defined as 
\[
\wt f(\la)=\int_S f(x)\, \phi_{\la}(x) \wrt \mu(x)\,.
\]
Note that $\wt f$ is an even function, as $\phi_{\la} = \phi_{-\la}$; moreover, for ``nice'' radial functions $f$ the following inversion formula holds:
\[
f(x) = C  \ioty \wt f(\la) \, \phi_{\la}(x)\, |\bc(\la)|^{-2} \wrt \la\,,
\]
where the density $|\bc(\la)|^{-2}$ is expressed in terms of the Harish-Chandra function
\begin{equation}\label{f: HC}
\bc(\la) 
= \frac{2^{Q-2i\la} \Gamma(2i\la) }{\Gamma(Q/2+i\la)} \,\frac{\Gamma(n/2)}{\Gamma((n-Q)/2+i\la)}
= \overline{\bc(-\overline{\la})}
\end{equation}
(see \cite[p.\ 151]{As}).

In the following statement we collect a few well-known properties of the function $\bc$ and the spherical functions $\phi_\lambda$, that will be of use in the sequel.

\begin{lemma}\label{l: sphestimates}
The following hold.
\begin{enumerate}[label=(\roman*)]
\item\label{en:sphestimates_cholo}
 The function $\bc^{-1}$ is holomorphic in the half-plane $\{\la: \Im\la<Q/2\}$.

\item\label{en:sphestimates_csymb}
 For all $k\in\BN$ and all $\sigma>0$, there exists a constant $C$ such that
\[
|\partial_{\la}^k\big(\la^{-1} \bc(-\la)^{-1}  \big) |
\leq C\, (1+\left|\Re\la\right|)^{(n-3)/2-k}
\]
for all $\lambda \in \BC$ with $0\leq \Im \la\leq \sigma$.

\item\label{en:sphestimates_sphbdd}
$\phi_\lambda$ is bounded if and only if $\bigmod{\Im\la} \leq Q/2$, and in that case $\|\phi_\lambda\|_\infty \leq 1$.

\item\label{en:sphestimates_sphone}
$\phi_{\pm iQ/2} \equiv 1$.

\item\label{en:sphestimates_sphexp}
For every $\lambda\in\BC \setminus \frac12 i \BZ$ and $r > 0$,
\[
\phi_{\lambda}(r) =  \bc(\lambda) \Phi_{\lambda}(r) + \bc(-\lambda) \Phi_{-\lambda}(r),
\]
where
\[
\Phi_\lambda(r) = \sum_{\ell=0}^{\infty} \Gamma_{\ell}(\lambda)  \,  \e^{(i\lambda-Q/2-\ell)r};
\]
here $\Gamma_{0}\equiv 1$, and moreover there exist $d,C>0$ such that
\[
|\Gamma_{\ell}(\lambda)| \leq C (1+\ell)^d
\]
for every $\ell \in \BN$ and every $\lambda$ in
\[
\{\lambda\in\BC \,:\, \Im \lambda \geq 0\}\cup\{ \lambda\in\BC \,:\, \left|\Im\lambda\right|\leq \left|\Re\lambda\right|\}.
\]
\end{enumerate}
\end{lemma}

\begin{proof}
Parts \ref{en:sphestimates_cholo} and \ref{en:sphestimates_csymb} follow from \eqref{f: HC}; in particular, for part \ref{en:sphestimates_csymb}, see the proof of \cite[Proposition A1(b)]{I2}. 
Part \ref{en:sphestimates_sphbdd} is in \cite[Theorem 5.12]{DR}, while part \ref{en:sphestimates_sphone} is clear from \eqref{f: sph_eigen}.
As for part \ref{en:sphestimates_sphexp}, see \cite[formula (5) and Theorem 3.2]{As}.
\end{proof}

It is useful to recall that on the group $S$ the convolution of two functions is defined by  
\[
f\ast g(x)= \int_S f(y)\, g(y^{-1}x)\wrt \mu(y)\,.
\]
Since the Laplace--Beltrami operator $\cL$ on $S$ is left-invariant, its heat and Poisson semigroups $\cH_t$ and $\cP_t$ can be realised by the right convolution with the heat and Poisson kernels $h_t$ and $p_t$, i.e.
\[
\cH_t f = f\ast h_t, \qquad \cP_t f = f \ast p_t.
\]
As we shall see, spherical analysis on $S$ allows us to obtain precise estimates for the heat and Poisson kernels, which will provide a sharper picture with respect to inclusions and noninclusions of various Hardy type spaces.

\subsection{Estimates for the heat maximal operator on \texorpdfstring{$S$}{S}}
We start with a technical lemma, which will be useful in our estimates for the heat kernel on $S$.

\begin{lemma}\label{l: sup_estimate}
Let $a \in (0,\infty)$.
\begin{enumerate}[label=(\roman*)]
\item\label{en:sup_estimate_sup}
If $\vep,k \in [0,\infty)$, then
\[
\sup_{t \geq 1} t^{-\vep} (1+t^{-1/2}|r-at|)^{-k} \asymp r^{-\min\{\vep,k\}}.
\]
for all $r \in [1,\infty)$. The implicit constants may depend on $\vep$, $a$ and $k$, but not on $r$.
\item\label{en:sup_estimate_int}
If $k>1$, then
\[
\int_1^{\infty} (1+t^{-1/2}|r-at|)^{-k} \wrt r \leq C \sqrt{t}.
\]
\end{enumerate}
\end{lemma}
\begin{proof}
We first prove part \ref{en:sup_estimate_sup}. Note that, if $r \leq a+1$, then, trivially,
\[
\sup_{t \geq 1} t^{-\vep} (1+t^{-1/2}|r-at|)^{-k} \asymp 1 \asymp r^{-\min\{\vep,k\}}.
\]
So we only need to consider the range where $r \geq a+1$.

Assuming that $r \geq a+1$, we first observe that
\[
\sup_{t \geq r/(a+1)} t^{-\vep} (1+t^{-1/2}|r-at|)^{-k} \asymp r^{-\vep}.
\]
Indeed, the lower bound is immediately obtained by taking $t=r/a$, while the upper bound simply follows by the majorisation $(1+t^{-1/2}|r-at|)^{-k} \leq 1$.

Let us now consider the range $1 \leq t \leq r/(a+1)$. Here $|r-at| \asymp r$, and $r/t^{1/2} \geq r/t \geq 1$, so 
\[
t^{-\vep} (1+t^{-1/2}|r-at|)^{-k} \asymp t^{-\vep+k/2} r^{-k} 
\]
and therefore
\[
\sup_{1 \leq t \leq r/(a+1)} t^{-\vep} (1+t^{-1/2}|r-at|)^{-k} 
\asymp \begin{cases}
r^{-k/2-\vep} &\text{if } k \geq 2\vep,\\
r^{-k} &\text{if } k \leq 2\vep.
\end{cases}
\]

The desired estimate follows by combining the ones above, and the proof of part \ref{en:sup_estimate_sup} is complete.

As for part \ref{en:sup_estimate_int}, we simply have
\[
\begin{split}
\int_1^\infty  ( 1+ t^{-1/2} | r-at | )^{-k} \wrt r &= \int_{|r-at|<\sqrt{t}} + \int_{|r-at| \geq \sqrt{t}} \\
&\leq 2\sqrt{t} + t^{k/2} \int_{|r| \geq \sqrt{t}} |r|^{-k} \wrt r \\
&\leq C \sqrt{t}, 
\end{split}
\]
provided $k > 1$.
\end{proof}

Using spherical Fourier analysis, we can now deduce a number of precise estimates for the heat kernel $h_t$.
Related results on any symmetric space of the noncompact type may be found in \cite{A1}.

\begin{proposition}\label{p: heat}
Suppose that $\ga$ and $c$ are real numbers and $\ga \geq 0$.
\begin{enumerate}[label=(\roman*)]
\item\label{en:heat_lower} 
For each radial compactly supported distribution $\psi$ on $S$ such that $\wt \psi(iQ/2) \neq 0$
there exists $\bar r \geq 1$ such that, for all $c \in \BR$,
\begin{equation}\label{f: pointwiselowerbd}
\sup_{t\geq 1} t^c |\cL^\gamma (\cH_t \psi)(x)|
\geq C \, \frac{\e^{-Q\mod{x}}}{\mod{x}^{(\gamma+1)/2-c}}
\quant x \in B_{\bar r}(o)^\compl.
\end{equation}
In particular, $\sup_{t\geq 1} t^c |\cL^\gamma (\cH_t \psi)| \notin L^1(S)$ whenever $c \geq (\gamma-1)/2$.

\item\label{en:heat_upper}
If $c \leq (\gamma+1)/2$, then the bound
\begin{equation}\label{f: pointwiseupperbd}
\sup_{t\geq 1} t^c |\cL^\gamma h_t(x)|\leq C\,\frac{\e^{-Q|x|}}{(1+|x|)^{(\gamma+1)/2-c}}\qquad \forall x\in S
\end{equation}
holds whenever $\ga \in \BN$ or $c > -(\ga+1)/2$.
\item\label{en:heat_integrability}
If $c < (\gamma-1)/2$, then $\sup_{t\geq 1} t^c |\cL^\gamma h_t| \in L^1(S)$.
\item\label{en:heat_l1}
For all $t>0$,
\[
\bignorm{\cL^\gamma h_t}{L^1}
\asymp 
\begin{cases}
t^{-\gamma} & \text{if } t \leq 1, \\
t^{-\gamma/2} & \text{if } t > 1.
\end{cases}
\]
\end{enumerate}
\end{proposition}

\begin{remark} 
If $\psi \in L^1(S)$ is radial, then $\wt \psi(iQ/2) = \int_S \psi \,d \mu$ (see Lemma \ref{l: sphestimates}~\ref{en:sphestimates_sphone}), so the assumption $\wt \psi(iQ/2) \neq 0$ in part \ref{en:heat_lower} above has a clear ``geometric'' meaning. On the other hand, the pointwise lower bound in part \ref{en:heat_lower} applies not only to (compactly supported) radial functions $\psi \in L^1(S)$, but also more generally to radial distributions, such as the Dirac delta $\delta_o$. Indeed, in the case where $\psi = \delta_o$, the upper bound in part \ref{en:heat_upper} shows  (for certain ranges of the parameters $\gamma$ and $c$) the optimality of the lower bound in part \ref{en:heat_lower}.
\end{remark}

\begin{proof}
We first prove part \ref{en:heat_lower}.
Let us immediately observe that the nonintegrability of $\sup_{t\geq 1} t^c |\cL^\gamma (\cH_t \psi)|$ when $c \geq (\gamma-1)/2$ follows from the pointwise estimate \eqref{f: pointwiselowerbd} and the integration formula \eqref{f: misuraDR} for radial functions. Thus, we are reduced to proving  \eqref{f: pointwiselowerbd} for any $c \in \BR$.

Up to replacing $\psi$ with $\psi/\wt\psi(iQ/2)$, we may and shall assume that $\wt\psi(iQ/2) = 1$.
Observe that, by spherical Fourier analysis, for all $r \geq 0$,
\begin{equation} \label{f: heat on cLpsi orig}
\cL^\gamma\big(\cH_t \psi \big) (r)
=  \ir  \wt \psi (\la) \, \big(\la^2+Q^2/4\big)^\gamma
             \, \e^{-t (\la^2+Q^2/4)} \,  \phi_{\la}({r})\,|\bc(\la)|^{-2} \wrt \la 
\end{equation}
and, if $r>0$, by Lemma \ref{l: sphestimates}~\ref{en:sphestimates_sphexp},
\begin{equation} \label{f: heat on cLpsi}
\cL^\gamma\big(\cH_t \psi \big) (r)
= 2\, \ir  \wt \psi (\la) \, \big(\la^2+Q^2/4\big)^\gamma
             \, \e^{-t (\la^2+Q^2/4)} \,  \Phi_{\la}({r})\,  \bc(-\lambda)^{-1} \wrt \la\,, 
\end{equation}
where we have used the fact that $\wt \psi$ is even. 

Assume that $r\geq 1$ and $t\geq 1$.
By the Paley--Wiener theorem for the spherical Fourier transform \cite[Theorem 3.6]{PS}, since $\psi$ is a compactly supported distribution,
$\wt \psi$ extends to an entire function of exponential type. 
In particular, there exist $m,R>0$ such that
\begin{equation}\label{f: crescitapsitilde}
	\bigmod{\wt \psi(\la)}
	\leq C \, 
	(1+\mod{\la})^{m} \, \e^{R \,\mod{\Im{\la}} }\qquad \forall \la\in \BC\,.
\end{equation}
Additionally, we can write
\[
\big(\la^2+Q^2/4\big)^\gamma = \big(Q/2+i\lambda\big)^\gamma \big(Q/2-i\lambda\big)^\gamma
\]
for all $\lambda \in \BC$ with $\mod{\Im \lambda} \leq Q/2$ and $\lambda \neq \pm iQ/2$; here we choose the branch of $w \mapsto w^\gamma$ on the complex plane slit along the half-line $(-\infty,0]$ that agrees with the arithmetic $\gamma$ power on $(0,\infty)$.
By Cauchy's theorem, we may shift the contour of integration
 in \eqref{f: heat on cLpsi}
from the real line to $\BR + i Q/2$, and obtain that  
\begin{equation} \label{f: heat on cLpsi contour}
\cL^\gamma \big(\cH_t \psi \big) (r)
=  2\, \ir \wt \psi (\la+iQ/2) \, (i\la)^\gamma \, p(\lambda)
             \, \e^{-t (\la^2+i\lambda Q)} \,  \Phi_{\la+iQ/2}({r})\, \wrt \la \,,
\end{equation}
where $p(\la)= \big(Q-i\la\big)^\gamma \, \bc(-\lambda-iQ/2)^{-1}$.
In particular, from Lemma \ref{l: sphestimates}~\ref{en:sphestimates_csymb} we deduce that the function $p$, thought of as a function of a real variable, belongs to a H\"ormander symbol class, namely,
\begin{equation}\label{f: pezzopol}
p \in S^M, \quad \text{where } M = \gamma + (n-1)/2.
\end{equation}

By Lemma \ref{l: sphestimates}~\ref{en:sphestimates_sphexp}, we can write 
\[
\e^{Qr} \Phi_{\la+iQ/2}({r})=\e^{i\lambda r}+R_{\lambda}(r)\,,
\]
where
\begin{equation}\label{f: expdecay}
|R_{\lambda}(r)|\leq C_\delta \, \e^{-\delta r}
\end{equation}
for all $r \geq 1$, $\delta\in (0,1)$ and $\lambda \in \BR$. Hence
\begin{equation}\label{f: splitIJ}
\begin{split}
\frac{1}{2} \e^{Qr} \cL^\ga \big(\cH_t \psi \big) (r)
& =  \ir  \wt \psi (\la+iQ/2) \, (i\la)^\gamma \, p(\la)
             \, \e^{-t \la^2} \, \e^{i\lambda (r-tQ)}  \wrt \la \\
&+ \ir \wt \psi (\la+iQ/2) \, (i\la)^\gamma \, p(\la)
             \, \e^{-t (\la^2+i\lambda Q)} \,  R_{\lambda}(r)  \wrt \la \\
&=I(t,r)+J(t,r)\,.
\end{split}
\end{equation}

Now, from \eqref{f: crescitapsitilde}, \eqref{f: pezzopol} and \eqref{f: expdecay} we deduce that
\begin{equation}\label{f: est_J_heat}
|J(t,r)| 
	\leq C_\delta \, \e^{-\delta r} \ir |\la|^\gamma (1+|\la|)^{M+m} 
	\, \e^{-t \,\la^2} \wrt\la\leq C_\delta \, t^{-(1+\gamma)/2} \e^{-\delta r} \,,
\end{equation}
for every $t\geq 1$ and $\delta\in (0,1)$.

Next, we show that the main contribution in $I(t,r)$ comes from integration in a neighbourhood of 
the origin.  Indeed, denote by $\eta$ a smooth even function on $\BR$ with support contained in
$[-2,2]$ such that $0\leq \eta \leq 1$ and $\eta$ is equal to $1$ in $[-1,1]$. Then, we may write $I(t,r)$ as $I_\eta(t,r) + I_{1-\eta}(t,r)$, where 
\begin{align*}
I_\eta(t,r)
&=  \ir \eta(\la) \, \wt \psi (\la+iQ/2) \, (i\la)^\gamma \, p(\lambda)
             \, \e^{-t \la^2} \,  \e^{i\lambda (r-tQ)} \wrt \la \,, \\
I_{1-\eta}(t,r)
&=  \ir \big(1-\eta(\la)\big) \wt \psi (\la+iQ/2) \, (i\la)^\gamma \, p(\lambda)
             \, \e^{-t \la^2} \,  \e^{i\lambda (r-tQ)}  \wrt \la \,.
\end{align*}
Now, by \eqref{f: crescitapsitilde} and \eqref{f: pezzopol},
\begin{equation}\label{f: est_Ine_heat}
\begin{split}
\bigmod{I_{1-\eta}(t,r)}
& \leq C\, \int_{\mod{\la} \geq 1}  \, \e^{-t \la^2}   (1+|\lambda|)^{M+m+\gamma}  \wrt \la \\
& \leq     C\, \e^{-t/2}\,.
\end{split}
\end{equation}
Next, we estimate $I_\eta(t,r)$. 
Since $\lambda \mapsto \wt\psi(\la+iQ/2) p(\lambda)$ is smooth near $\lambda = 0$ and $\wt\psi(iQ/2)=1$, we deduce that
\[
\wt\psi(\la+iQ/2) p(\la) = p(0) + O(|\la|)
\]
for $|\lambda| \leq 2$.
Therefore
\[
\begin{split}
I_\eta(t,r)
&= 
p(0)  \ir \eta(\la)  \, (i\la)^\gamma  
             \, \e^{-t \la^2} \,  \e^{i\lambda (r-tQ)}  \wrt \la 
             +  \int_{-2}^{2}  \e^{-t \la^2}   O(|\la|^{1+\gamma})
              \wrt \la \\
             &=p(0) \,I_{\eta}^0(t,r) + I^1_\eta(t,r),
\end{split}
\]
and clearly
\begin{equation}\label{f: est_Ie1_heat}
|I_{\eta}^1(t,r)|\leq C \int_{-2}^{2} |\la|^{1+\gamma}
             \, \e^{-t \la^2}    \wrt \la 
						\leq C \, t^{-(2+\gamma)/2}\,.
\end{equation}
Since $p(0) = Q^\gamma \, \bc(-iQ/2)^{-1} \neq 0$ by \eqref{f: HC}, it remains to consider $I_{\eta}^0(t,r)$. 

Set $\bar{t}(r)=\frac{r-A\sqrt r}{Q}$, where $A \in \BR$ is a constant to be determined.
Clearly, for all $r$ sufficiently large, we have that $\bar t(r) \geq 1$, and
moreover
\[
\begin{split}
I_\eta^0(\bar t(r),r)
&= \ir \eta(\la)  \, (i\la)^\gamma  \, \e^{-\bar{t}(r) \la^2} \,  \e^{iA\lambda \sqrt r}  \wrt \la \\
&= r^{-(1+\gamma)/2} \,\ir \eta(v/{\sqrt r})   \, \e^{-\bar{t}(r) v^2/r} \,  \e^{iAv}  (iv)^\gamma \wrt v\,,
\end{split}
\]
where we used the change of variables $\la \sqrt r = v$.
Notice that 
\[
\frac{\bar{t}(r)}{r} = \frac{1}{Q} - \frac{A}{Q\sqrt r}, 
\]
which tends to $1/Q$
as $r$ tends to infinity.  An application of the dominated convergence theorem shows that
\[
\lim_{r \to \infty} \ir \eta(v/{\sqrt r})   
             \, \e^{-\bar{t}(r) v^2/r} \,  \e^{iAv}  (iv)^\gamma \wrt v 
						= \, \ir  \e^{iAv} \, \e^{-v^2/Q} \, (iv)^\gamma \wrt v ,
\]
as $\eta (0) = 1$;
the last integral is the value at $-A$ of the Fourier transform of 
$v \mapsto \e^{-v^2/Q} (iv)^\gamma$, so we can choose $A$ such that the integral does not vanish. With this choice of $A$, 
there exist $C, r_0 \geq 1$ such that, for all $r \geq  r_0$, we have $\bar t(r) \geq 1$ and
\[
|I_{\eta}^0(\bar t(r),r)| \geq C r^{-(1+\gamma)/2} .
\]
Note, on the other hand, that, by the previous estimates \eqref{f: est_J_heat}, \eqref{f: est_Ine_heat} and \eqref{f: est_Ie1_heat},
\[
|I_{\eta}^1(\bar t(r),r)|, |J(\bar t(r),r)|, |I_{1-\eta}(\bar t(r),r)| \leq C r^{-(2+\gamma)/2},
\]
since $\bar t(r) \asymp r$ for $r$ large.
By combining the above estimates, we conclude that, for $r$ sufficiently large,
\[
\sup_{t>1} t^c \bigmod{\cL^\gamma (\cH_t \psi) (r)}
\geq \bar{t}(r)^c  \bigmod{\cL^\gamma (\cH_{\bar{t}(r)} \psi) (r)} 
\geq C r^{c-(\gamma+1)/2} \, \e^{-Q r},
\]
which proves \eqref{f: pointwiselowerbd} and thus part \ref{en:heat_lower}.
 
We now prove part \ref{en:heat_upper}, that is, the pointwise upper bound \eqref{f: pointwiseupperbd}. Note that $\cL^\gamma h_t =\cL^\gamma(\cH_t \delta_o)$, hence all the above computations can be applied with $\psi=\delta_o$. In particular, by \eqref{f: heat on cLpsi orig} and 
Lemma \ref{l: sphestimates}~\ref{en:sphestimates_csymb}-\ref{en:sphestimates_sphbdd}, for every $t\geq 1$,
\begin{equation} \label{f: rsmall}
|\cL^\gamma h_t(r)|
						\leq C \, \e^{-tQ^2/4} 
						\,,
\end{equation}
where $C$ does not depend on $t$.  This proves \eqref{f: pointwiseupperbd} for $x\in B_{1}(o)$, irrespective of the values of $\gamma$ and $c$.

Take now $r \geq 1$. Then, by \eqref{f: splitIJ},
\begin{equation} \label{f: cLht}
\begin{split}
\frac{1}{2} \e^{Qr} \cL^\gamma h_t   (r)
& = \ir  (i\la)^\gamma \, p(\la) \, \e^{-t \la^2}  \e^{i\lambda (r-tQ)} \wrt \la \\
&+ \ir (i\la)^\gamma \, p(\la) \, \e^{-t (\la^2+i\lambda Q)} \,  R_{\lambda}(r)  \wrt \la \\
&=I(t,r)+J(t,r)\,,
\end{split}
\end{equation}
where, as before, by \eqref{f: est_J_heat},
\begin{equation}\label{f: estJ_heat_bis}
|J(t,r)| \leq C_\delta \,  t^{-(1+\gamma)/2} \e^{-\delta r},
\end{equation}
for every $t\geq 1$ and $\delta \in (0,1)$. 
Moreover, as $p \in S^M$ (see \eqref{f: pezzopol}), an estimate for $I(t,r)$ is immediately obtained by taking absolute values inside the integral:
\begin{equation}\label{f: basic_est_I}
|I(t,r)| \leq C \ir |\la|^\gamma (1+|\lambda|)^M \e^{-t\lambda^2} \wrt\la \leq C \, t^{-(1+\gamma)/2}
\end{equation}
for every $r,t \geq 1$. The above estimates are enough to conclude the proof of \eqref{f: pointwiseupperbd} when $c=(\gamma+1)/2$. However, when $c<(\gamma+1)/2$,
the estimate for $I(t,r)$ in \eqref{f: basic_est_I} is not enough, as it does not give any decay in $r$.
To improve on this, we will use integration by parts in order to take advantage of the oscillatory term $\e^{i\lambda (r-tQ)}$.

Let us first assume that $\gamma \in \BN$.
If we integrate by parts $k$ times in the expression for $I(t,r)$, we obtain
\begin{equation}\label{f: int_byparts}
I(t,r)= \frac{i^k}{(r-Qt)^k} \, \ir \partial_\la^k [ (i\la)^\gamma p(\la) \, \e^{-t \la^2}] \,  \e^{i\lambda(r-tQ)}  \wrt \la.
\end{equation}
Recall from \eqref{f: pezzopol} that $p \in S^M$. Hence, by arguing inductively,
one readily shows that, for every nonnegative integer $k$, the $k$th derivative $\partial_\la^k [\la^\gamma\,  p(\la) \, \e^{-t \la^2}]$ 
is a finite sum of terms of the form $t^a\, (i\la)^b \,q(\la) \,\e^{-t\lambda^2}$, where
\[
a,b,\ell \in \BN, \quad q \in S^{M-\ell} \quad\text{and}\quad b-2a-\ell = \gamma-k;
\]
in particular $a -b/2 \leq (k-\gamma)/2$ and
\begin{equation}\label{f: int_est_byparts}
\begin{split}
\int_\BR \bigmod{t^a\, (i\la)^b \,q(\la) \,\e^{-t\lambda^2}} \wrt\lambda &\leq C t^{a-b/2} \, \int_\BR (t \la^2)^{b/2} (1+|\la|)^{M-\ell} \e^{-t\la^2} \wrt\la \\
&\leq C t^{(k-\gamma-1)/2}
\end{split}
\end{equation}
for all $t \geq 1$.
Therefore, from \eqref{f: int_byparts}, we deduce that
\begin{equation}\label{f: byparts_est_I}
|I(t,r)| \leq C_k \, |r-tQ|^{-k} \, t^{(k-\gamma-1)/2}.
\end{equation}
By combining the previous estimates \eqref{f: estJ_heat_bis}, \eqref{f: basic_est_I} and \eqref{f: byparts_est_I}, we conclude that, for all $k\in\BN$ and $\delta \in (0,1)$ there exists a positive constant $C$ such that
\begin{equation}\label{f: lapheatestimate}
|\cL^\gamma h_t(r)|
\leq C\,\e^{-Qr}\,t^{-(\gamma+1)/2} \Big[ \e^{-\delta r} +  \big( 1+ t^{-1/2} \bigl| r-Qt\bigr| \big)^{-k} \Big]
\end{equation}
for all $t,r\geq 1$. Consequently, from Lemma \ref{l: sup_estimate}~\ref{en:sup_estimate_sup} we deduce that, for all $c \leq (\gamma+1)/2$ and $r \geq 1$,
\begin{equation}\label{f: lapheatmaxestimate}
\sup_{t \geq 1} t^c \, |\cL^\gamma h_t(r)| \leq C e^{-Qr} r^{-\min\{k,(1+\gamma)/2-c\}}.
\end{equation}
As in this case we can pick any $k \in \BN$, the above estimate completes the proof of the upper bound \eqref{f: pointwiseupperbd} and of part \ref{en:heat_upper} in the case where $\gamma \in \BN$.

Suppose instead that $\gamma$ is not an integer. In this case, we cannot integrate by parts in \eqref{f: int_byparts} arbitrarily many times, as repeated differentiation of the term $(i\lambda)^\gamma$ may produce a non-integrable singularity at $0$. Nevertheless, we can integrate by parts at least $\lceil \gamma \rceil$ times (note that $|\lambda|^{\gamma-\lceil \gamma\rceil }$ is locally integrable at $0$, as $\gamma-\lceil \gamma \rceil > -1$), and the previous argument can still be run (indeed, the estimate \eqref{f: int_est_byparts} remains true even for noninteger $b>-1$), thus \eqref{f: lapheatestimate} and \eqref{f: lapheatmaxestimate} are valid for $k = \lceil \gamma \rceil$. As a consequence, we may still deduce the upper bound \eqref{f: pointwiseupperbd} whenever $\lceil \gamma\rceil \geq (\gamma+1)/2-c$, that is, $c \geq (\gamma+1)/2 - \lceil \gamma \rceil$.

To further extend the range of validity of \eqref{f: pointwiseupperbd} and complete the proof of part \ref{en:heat_upper}, we make use of complex interpolation. Namely, the above argument leading to \eqref{f: lapheatestimate} can be applied, \emph{mutatis mutandis}, when the real exponent $\gamma$ is replaced by the complex exponent $z=\gamma+i\theta$, with $\gamma \in (-1,\infty)$ and $\theta \in \BR$. In this case, we can integrate by parts $\lceil \gamma \rceil$ times and, instead of \eqref{f: lapheatestimate}, we obtain, for any $\delta \in (0,1)$, the estimate
\begin{equation}\label{f: lapheatestimate_complex}
|\cL^z h_t(r)|
\leq C (1+|\theta|)^{\lceil \gamma \rceil} \, \e^{-Qr}\,t^{-(\gamma+1)/2} \Big[ \e^{-\delta r} +  \big( 1+ t^{-1/2} \bigl| r-Qt\bigr| \big)^{-\lceil\gamma\rceil} \Big]
\end{equation}
for all $t,r \geq 1$, where the constant $C$ does not depend on $\theta$. Consequently, for any $\vep,\delta \in (0,1)$, we deduce the estimate 
\begin{equation}\label{f: lapheatestimate_complex_interpol}
 \left|\e^{z^2} \cL^z h_t(r)\right|
\leq C \,\e^{-Qr}\,t^{-(\gamma+1)/2} \Big[ \e^{-\delta r} +  \big( 1+ t^{-1/2} \bigl| r-Qt\bigr| \big)^{\vep-\gamma-1} \Big]
\end{equation}
whenever $\gamma > -1$ and $\lceil \gamma \rceil - \gamma \geq 1-\vep$; again, the constant $C$ may depend on $\gamma,\vep,\delta$ but not on $\theta$. As $\cL^z h_t(x)$ is, for fixed $x$ and $t$, a holomorphic function of $z$ on the half-plane where $\Re z = \gamma > -1$ (cf.\ \eqref{f: heat on cLpsi orig}), an application of Hadamard's three-line theorem allows us to dispense with the constraint $\lceil\gamma\rceil - \gamma \geq 1-\vep$ and extend the validity of \eqref{f: lapheatestimate_complex_interpol} to all $\gamma > -1$: more precisely, for all $\gamma>-1$ and $\vep \in (0,1)$, there exists $\delta > 0$ such that
\eqref{f: lapheatestimate_complex_interpol} holds with $z=\gamma+i\theta$ for all $\theta \in \BR$ and $r,t \geq 1$.

If we now apply \eqref{f: lapheatestimate_complex_interpol} with $\theta=0$ and $\gamma \geq 0$, then Lemma \ref{l: sup_estimate}~\ref{en:sup_estimate_sup} yields, for all $\vep \in (0,1)$ and $c \leq (\gamma+1)/2$,
\begin{equation}\label{f: lapheatmaxestimate_interpol}
\sup_{t \geq 1} t^c \, |\cL^\gamma h_t(r)| \leq C e^{-Qr} r^{-\min\{1+\gamma-\vep,(1+\gamma)/2-c\}}.
\end{equation}
On the other hand, if $c> -(1+\gamma)/2$, then $(1+\gamma)/2-c < 1+\gamma$, so we can pick $\vep>0$ sufficiently small that $1+\gamma-\vep \geq (1+\gamma)/2-c$, and \eqref{f: lapheatmaxestimate_interpol} gives the desired upper bound \eqref{f: pointwiseupperbd} in full generality. This completes the proof of part \ref{en:heat_upper}.

We now prove part \ref{en:heat_integrability}. Let us first note that the integrability of $\sup_{t\geq 1} t^c |\cL^\gamma h_t|$ for $c<(\gamma-1)/2$ immediately follows from \eqref{f: misuraDR} and the pointwise bound \eqref{f: pointwiseupperbd}, provided the latter applies. Thus, in light of part \ref{en:heat_upper}, the aforementioned integrability does hold true whenever $\gamma$ is an integer or $c \in (-(\gamma+1)/2,(\gamma-1)/2)$. On the other hand, when $\gamma$ is not an integer, necessarily $\gamma > 0$, thus the interval $(-(\gamma+1)/2,(\gamma-1)/2)$ is not empty; as 
$\sup_{t\geq 1} t^c |\cL^\gamma h_t|$ is increasing in $c$, we conclude that the latter is integrable whenever $c < (\gamma-1)/2$, as required.

Finally, let us prove part \ref{en:heat_l1}.  We first prove upper estimates.  Since $\cL$ is a sectorial operator on $L^1$, by the moment inequality 
\cite[Proposition~6.6.4]{Haa} it is enough to consider the case where $\gamma \in \BN$. 
In this case, the estimate for $t \leq 1$ is known in greater generality (see, e.g., 
\cite[Theorem IX.1.3~(ii)]{VSC}, and note that $\cL^\gamma h_t = (-\partial_t)^\gamma h_t$).
If instead $t \geq 1$, then we can use the estimates from the proof of part \ref{en:heat_upper}. Namely, \eqref{f: rsmall} shows that
\[
\int_{B_1(o)} |\cL^\gamma h_t(x)| \wrt \mu(x) 
\leq C \, \e^{-t Q^2/4},
\]
thus the main contribution comes from integration over $B_1(o)^\compl$. From \eqref{f: lapheatestimate} and \eqref{f: misuraDR}, 
instead, we deduce that, for all $k \in \BN$ with $k>1$,
\[
\int_{B_1(o)^\compl} |\cL^\gamma h_t(x)| \wrt \mu(x) 
\leq \frac{C}{t^{(1+\gamma)/2}} \Big[ 1 + \int_1^\infty  \Big( 1+ t^{-1/2} \bigl| r-Qt\bigr| \Big)^{-k} \wrt r \Big] \leq C t^{-\gamma/2},
\]
where we applied Lemma \ref{l: sup_estimate}~\ref{en:sup_estimate_int} in the last inequality,
and the required upper estimate follows.

It remains to prove matching lower estimates.  From Lemma \ref{l: sphestimates}~\ref{en:sphestimates_sphbdd} it follows that, if $f$ is in $\lu{S}$, then 
\[
	\sup_{\la\in \bT_1} \, \bigmod{\wt f(\la)} 
	\leq \bignorm{f}{L^1(S)},
\]
where $\bT_1 \defeq \big\{ \la \in \BC: \bigmod{\Im{\la}} \leq Q/2\}$. As
\[
\big(\cL^\ga h_t\big)\wt{\phantom a}(\la)
= (\la^2+Q^2/4)^\ga \, \e^{-t(\la^2+Q^2/4)},
\]
we easily deduce that
\[
\|\cL^\ga h_t\|_{L^1} \geq \bigmod{\big(\cL^\ga h_t\big)\wt{\phantom a}(t^{-1/2}+iQ/2)} \asymp t^{-\gamma/2}
\]
for $t \geq 1$, and
\[
\|\cL^\ga h_t\|_{L^1} \geq \bigmod{\big(\cL^\ga h_t\big)\wt{\phantom a}(t^{-1/2})} \asymp t^{-\gamma}
\]
for $0 < t \leq 1$.
This concludes the proof of part \ref{en:heat_l1} and of the proposition.  
\end{proof}

The above estimates imply a number of inclusions and noninclusions between the spaces $\gXga{S}$ and $H^1_{\cH,c}(S)$, 
complementing those in Theorem \ref{t: maximal heat manifold}.

\begin{theorem} \label{t: inclusion heat}
Let
$\gamma \in (0,\infty)$ and $c \in \BR$.
\begin{enumerate}[label=(\roman*)]
\item\label{en:inclusion_heat_not}
The space $\fX^{2c+1}(S)$ is not included in $H^1_{\cH,c}(S)$ whenever $2c+1>0$.
\item\label{en:inclusion_heat_yes}
The space $\gXga{S}$ is included in $H^1_{\cH,c}(S)$ whenever $\gamma>2c+1$. 
\item\label{en:inclusion_heat_ghu}
$\ghu{S} = H^1_{\cH,c}(S)$ for all $c <-1/2$, but $\ghu{S} \not\subseteq H^1_{\cH,-1/2}(S)$.
\end{enumerate}
\end{theorem}

\begin{proof}
Let us prove part \ref{en:inclusion_heat_not}. Let $\delta=2c+1>0$. Take any radial nonnegative $\psi \in C^\infty_c(S)$, which is not identically zero; 
in particular $\wt \psi(iQ/2) = \int_S \psi \wrt\mu > 0$ (see Lemma \ref{l: sphestimates}~\ref{en:sphestimates_sphone}),
and moreover $\cL^\delta \psi \in \fX^\delta(S)$ by Proposition \ref{p: cinftyc}~\ref{en:cinftyc_cinftyc}. On the other hand, by
Proposition \ref{p: heat}~\ref{en:heat_lower}, 
$\maxHc(\cL^\delta \psi)$ is not in $L^1(S)$, whence $\cL^\delta \psi \notin H^1_{\cH,c}(S)$.

Let us now prove part \ref{en:inclusion_heat_yes}.
If $f \in \gXga{S}$, then, by Proposition \ref{p: cinftyc}~\ref{en:cinftyc_car}, $f \in \ghu{S}$ and $\cL^{-\gamma} f \in \ghu{S}$, with
\[
\|f\|_{\frh^1} + \|\cL^{-\gamma} f\|_{\frh^1} \leq C \|f\|_{\fX^\gamma}.
\]
In particular, by Proposition \ref{p: H1Hloc},
\[
\|f\|_{L^1} + \| \maxHloc f\|_{L^1} \leq C \|f\|_{\frh^1} \leq C \|f\|_{\fX^\gamma}.
\]
On the other hand,
\[
\cH_t f = (\cL^{-\gamma} f) * (\cL^\gamma h_t),
\]
whence
\[
\sup_{t\geq 1} t^c |\cH_t f| \leq |\cL^{-\gamma} f| * \sup_{t \geq 1} t^c |\cL^\gamma h_t|.
\]
By Proposition \ref{p: heat}~\ref{en:heat_integrability}, the latter supremum is integrable whenever $c<(\gamma-1)/2$; so, by Young's inequality,
\[
\left\|\sup_{t\geq 1} t^c |\cH_t f|\right\|_{L^1} \leq C \| \cL^{-\gamma} f\|_{L^1} \leq C  \| \cL^{-\gamma} f\|_{\frh^1} \leq C  \|f\|_{\fX^\gamma},
\]
as required.

As for part \ref{en:inclusion_heat_ghu}, the inclusion $H^1_{\cH,c}(S) \subseteq \ghu{S}$ follows from Theorem \ref{t: maximal heat manifold}~\ref{en:maxhpman_heat}; the remaining inclusion and noninclusion results are proved as in parts \ref{en:inclusion_heat_not} and \ref{en:inclusion_heat_yes}, by taking $\delta=\gamma = 0$.
\end{proof}

\begin{remark}
Theorems \ref{t: maximal heat manifold}~\ref{en:maxhpman_heat} and \ref{t: inclusion heat}~\ref{en:inclusion_heat_yes} imply that, for all $\gamma,\vep > 0$,
\[
H^1_{\cH,\gamma+\vep}(S) \subset \gXga{S} \subset H^1_{\cH,(\gamma-1)/2-\vep}(S).
\]
It would be interesting to know whether these inclusion ranges are sharp (this is the case for the second inclusion in the series, 
but we do not know about the first). If this were the case, we would have a substantially different situation compared to the case 
of the Poisson maximal operator discussed in Corollary
\ref{c: maximal heat manifold} above
(note that the gap between $(\gamma-1)/2$ 
and $\gamma$ increases with $\gamma$). On the other hand, in the ``limit case'' $\gamma = 0$ 
(with $\gXga{S}$ replaced by $\ghu{S}$), the first inclusion is certainly not sharp, given that 
$\ghu{S} = H^1_{\cH,-1/2-\vep}(S)$ for all $\vep > 0$.
\end{remark}

\subsection{Estimates of the Poisson maximal operator on \texorpdfstring{$S$}{S}} 

The following result is the analogue for the semigroups $\{\e^{-t\cL^\al}: t\geq 0\}$ of Proposition~\ref{p: heat}~\ref{en:heat_lower}: the needed estimates for the Poisson semigroups correspond to the case where $\al=1/2$.  

\begin{proposition} \label{p: maximal Poisson H3}
Let $\al$, $c$ and $\ga$ be real numbers such that $0<\al<1$ and $\ga \geq 0$.
 Let $\psi$ be a radial distribution with compact support 
such that $\wt \psi (iQ/2) \neq 0$.  Then there exists a positive constant $C$ and $\bar r \geq 1$ such that 
\begin{equation}\label{f: lowerbd_poisson_max}
	\sup_{t\geq 1}\, t^c\, \bigmod{\e^{-t\cL^\al}\big(\cL^{\ga}\psi\big) (x)}
	\geq C \, \frac{\e^{-Q\mod{x}}}{\mod{x}^{1+\ga-c\al}}
\quant x \in B_{\bar r}(o)^\compl.
\end{equation}
In particular, $\sup_{t\geq 1}\, t^c\, \bigmod{\e^{-t\cL^\al}\big(\cL^{\ga}\psi\big)} \notin L^1(S)$ whenever $c \geq \ga/\al$.
\end{proposition}

\begin{proof}
We first observe that the nonintegrability of $\sup_{t\geq 1} t^c |\cL^\gamma (\e^{-t\cL^\al} \psi)|$ when $c \geq \gamma/\al$ follows from the pointwise estimate \eqref{f: lowerbd_poisson_max} and the integration formula \eqref{f: misuraDR}. Thus, we are reduced to proving the lower bound \eqref{f: lowerbd_poisson_max}.

The proof of the latter closely follows that of part \ref{en:heat_lower} of Proposition \ref{p: heat}. Indeed, much as in that proof, we may assume that $\wt \psi(iQ/2) = 1$ and write, for all $r>0$,
\[\begin{split}
\frac{1}{2} \cL^\gamma\big(\e^{-t\cL^\alpha} \psi \big) (r)
&= \ir  \wt \psi (\la) \, \big(\la^2+Q^2/4\big)^\gamma
             \, \e^{-t (\la^2+Q^2/4)^\alpha} \,  \Phi_{\la}({r})\,  \bc(-\lambda)^{-1} \wrt \la \\
&= \ir \wt \psi (\la+iQ/2) \, (i\la)^\gamma p(\la)
             \, \e^{-t (\la^2+iQ\la)^\alpha} \,  \Phi_{\la+iQ/2}({r})\,  \wrt \la
\end{split}\]
(the only difference being the exponent $\alpha$ in place of $1$); as before, we are using the branches of $w \mapsto w^\alpha$ and $w \mapsto w^\gamma$ on the complex plane slit along the half-line $(-\infty,0]$ that coincide with the arithmetic $\alpha$ and $\gamma$ powers on $(0,\infty)$.
Still following the proof of Proposition \ref{p: heat}~\ref{en:heat_lower}, we then proceed to split
\begin{align*}
\frac{1}{2} e^{Qr} \cL^\gamma\big(\e^{-t\cL^\alpha} \psi \big) (r) &= I(r,t) + J(r,t),\\
I(r,t) &= I_\eta(r,t) + I_{1-\eta}(r,t),\\
I_\eta(r,t) &= p(0) \, I_\eta^0(r,t) + I_{\eta}^1(r,t),
\end{align*}
where again $p(0) = Q^\gamma \, \bc(-iQ/2)^{-1} \neq 0$.

A substantial difference between the case $\alpha \in (0,1)$ considered here and the case $\alpha=1$ discussed in Proposition \ref{p: heat} is that, when $\alpha \in (0,1)$,
\[
\Re \left[(\la^2+i\la Q)^\al\right] = |\la|^\al (\la^2+Q^2)^{\al/2} \cos (\al \arctan(Q/\la)) \asymp |\la|^\al (1+|\la|)^\al
\]
for all $\lambda \in \BR$ (the implicit constants may depend on $\alpha$), and in particular
\begin{equation}\label{f: real part of the phase}
\bigmod{\e^{-t (\la^2+iQ\la)^\alpha}} \leq \e^{-\kappa t |\la|^\al}
\end{equation}
for some $\kappa>0$.
(In the case $\alpha=1$, we have $\bigmod{\e^{-t (\la^2+iQ\la)}} = \e^{- t |\la|^2}$ instead, with an exponent $2$ in place of $1$.)
Armed with the estimate \eqref{f: real part of the phase}, we can follow the arguments in the proof of Proposition \ref{p: heat}~\ref{en:heat_lower} to obtain that, for all $\delta \in (0,1)$,
\begin{equation}\label{f: poisson_errorterms}
\begin{aligned}
|J(r,t)| &\leq C_\delta \, t^{-(1+\ga)/\al} \e^{-\delta r},\\
|I_{1-\eta}(r,t)| &\leq C \, \e^{-\kappa t/2},\\
|I^1_{\eta}(r,t)| &\leq C \, t^{-(2+\ga)/\al}
\end{aligned}
\end{equation}
for all $r,t \geq 1$ (cf.\ \eqref{f: est_J_heat}, \eqref{f: est_Ine_heat} and \eqref{f: est_Ie1_heat}).

We are then left with estimating the main term
\[
I_\eta^0(r,t) = 
\ir \eta(\la)  \, (i\la)^\gamma  
             \, \e^{-t (\la^2+iQ\la)^\al} \,  \e^{i\lambda r}  \wrt \la \, ,
\]
where $\eta \in C_c^\infty(\BR)$ is even and supported in $[-2,2]$, with $0 \leq \eta \leq 1$ and $\eta|_{[-1,1]} \equiv 1$.
Let $\sigma > 0$ be a parameter to be fixed later, and notice that
\begin{equation}\label{f: Ieta0_convergence}
\begin{split}
I_\eta^0(r,(r/\si)^\al) &= 
\ir \eta(\la)  \, (i\la)^\gamma  
             \, \e^{-[r(\la^2+iQ\la)/\si]^\al} \,  \e^{i\lambda r}  \wrt \la \\
&=(\si/r)^{\ga+1} \ir \eta(\si v/r)  \, (i v)^\gamma  
             \, \e^{-(\si v^2/r+iQv)^\al} \,  \e^{i\si v}  \wrt v .
\end{split}
\end{equation}
By dominated convergence, and since $\eta(0)=1$,
one can show that, as $r \to +\infty$, the latter integral tends to
\begin{equation}\label{f: ft_int}
\ir (i v)^\gamma \, \e^{-(iQv)^\al} \,  \e^{i\si v}  \wrt v,
\end{equation}
that is, the value $\hat F(-\si)$ of the Fourier transform of $F : v \mapsto (i v)^\gamma \, \e^{-(iQv)^\al}$.

We now observe that, due to our choices of the branches of $w \mapsto w^\gamma$ and $w \mapsto w^\alpha$, $F$ admits a holomorphic extension to the lower half-plane $\Omega_- = \{ z \in \BC : \Im z < 0\}$; moreover, if $z = v+is$, $s<0$, then
\[
|F(z)| = |iv-s|^\ga \,\e^{-Q^\al \Re \left[(iv-s)^\al\right]} \leq |v|^\ga \,\e^{-\kappa Q^\al |v|^\al},
\]
and the latter expression is in $L^2(\BR)$ and independent of $s$. Consequently, $F$ is in the holomorphic Hardy space $H^2(\Omega_-)$ and, by the Paley--Wiener theorem, $\hat F|_{(0,\infty)} \equiv 0$. It follows that there must exist a $\si > 0$ such that $\hat F(-\si) \neq 0$, for otherwise $\hat F$ and $F$ would vanish identically.

Thus, with this choice of $\sigma$, the integral \eqref{f: ft_int} does not vanish, and therefore from \eqref{f: Ieta0_convergence} we deduce that there exists $r_0 \geq \max\{1,\si\}$ sufficiently large that
\[
|I_\eta^0(r,(r/\si)^\al)| \geq C \, r^{-(\ga+1)}.
\]
for all $r \geq r_0$. On the other hand, the previous estimates \eqref{f: poisson_errorterms} give
\[
|J(r,(r/\si)^\al)|,|I_{1-\eta}(r,(r/\si)^\al)|,|I^1_{\eta}(r,(r/\si)^\al)| \leq C \, r^{-(2+\ga)}.
\]
So we can conclude that, for $r \geq r_0$ sufficiently large, 
\[
\sup_{t \geq 1} t^c |\cL^\gamma(\e^{-t\cL^\al} \psi)(r)| \geq (r/\si)^{\al c} |\cL^\gamma(\e^{-(r/\si)^\al \cL^\al} \psi)(r)| \geq C \e^{-Qr} r^{c\al-\ga-1},
\]
as required.
\end{proof}

Thanks to the above estimate, we can now complement the inclusions of Corollary \ref{c: maximal heat manifold} with the following result, ruling out that the spaces $\gXga{S}$ and $H^1_{\cR}(S)$ can be characterised in terms of the weighted Poisson maximal function.

\begin{theorem}\label{t: comparison weighted poisson}
Let $\ga > 0$ and $c \in \BR$.
\begin{enumerate}[label=(\roman*)]
\item\label{en:cwp_Xnotin} $\gXga{S}$ is not included in $H^1_{\cP,2\ga}(S)$.
\item\label{en:cwp_X} $\gXga{S}$ and $H^1_{\cP,c}(S)$ are different spaces.
\item\label{en:cwp_R} $H^1_{\cR}(S)$ and $H^1_{\cP,c}(S)$ are different spaces.
\end{enumerate}
\end{theorem}
\begin{proof}
Part \ref{en:cwp_Xnotin} is proved in the same way as Theorem \ref{t: inclusion heat}~\ref{en:inclusion_heat_not}. Indeed, take a nontrivial, compactly supported, nonnegative smooth radial function $\psi$ on $S$. 
Then, by Proposition \ref{p: cinftyc}~\ref{en:cinftyc_cinftyc}, $\cL^\ga \psi \in \gXga{S}$. On the other hand, $\wt\psi(iQ/2) = \int_S \psi \wrt \mu \neq 0$.
Therefore, by Proposition \ref{p: maximal Poisson H3}, we conclude that $\cP_*^{2\ga} (\cL^\ga \psi) \notin L^1(S)$, that is, $\cL^\ga \psi \notin H^1_{\cP,2\ga}(S)$.

We now prove part \ref{en:cwp_X}. By Corollary \ref{c: maximal heat manifold}, we already know that $\gXga{S} \neq H^1_{\cP,c}(S)$ whenever $c \neq 2\gamma$. On the other hand, when $c=2\ga$, part \ref{en:cwp_Xnotin}
rules out that $\gXga{S}$ can coincide with $H^1_{\cP,2\ga}(S)$.

Finally, as $H^1_{\cR}(S) = \gXum{S}$ by Theorem~\ref{t: char Riesz}, part \ref{en:cwp_R} follows from part \ref{en:cwp_X}.
\end{proof}

Drawing on the theory developed above, we can finally conclude that the Riesz, Poisson and heat Hardy spaces on a Damek--Ricci space $S$ are pairwise different.

\begin{corollary} \label{c: comparison heat Poisson}
The following hold.
\begin{enumerate}[label=(\roman*)]
\item\label{en:comparison_heat_poisson_rp}
$\huR{S}$ is properly contained in $\huP{S}$.
\item\label{en:comparison_heat_poisson_hp}
$\huH{S}$ is properly contained in $\huP{S}$.
\item\label{en:comparison_heat_poisson_rh}
$\huR{S}$ is not contained in $\huH{S}$.
\end{enumerate}  
\end{corollary}

\begin{proof}
The proper containment in part \ref{en:comparison_heat_poisson_rp} is among those proved in Corollary \ref{c: maximal heat manifold} for any manifold of class $\cM$.

We now prove part \ref{en:comparison_heat_poisson_hp}.  
Observe that the containment is the case $c=0$ of Theorem \ref{t: maximal heat manifold}~\ref{en:maxhpman_heatpoisson}.
Now $\gXu{S}\subset \huP{S}$, by 
Corollary \ref{c: maximal heat manifold},
and 
$\gXu{S} \not\subseteq \huH{S}$, by Theorem~\ref{t: inclusion heat}~\ref{en:inclusion_heat_not}.  Therefore 
$\huH{S}$ cannot possibly coincide with $\huP{S}$.  This proves part \ref{en:comparison_heat_poisson_hp}.

Finally we prove part \ref{en:comparison_heat_poisson_rh}.  We already know that $\gXu{S} \subset  \gXum{S}=\huR{S}$ by Theorem~\ref{t: char Riesz},
and that $\gXu{S} \not\subseteq\huH{S}$.  Thus, $\huR{S}$ and $\huH{S}$ 
cannot possibly agree.  
\end{proof}

\begin{remark}
Entirely analogous results to those in Theorem \ref{t: maximal heat manifold}, Corollary \ref{c: maximal heat manifold} and Theorem \ref{t: comparison weighted poisson} hold when the weighted Poisson--Hardy spaces $\huPc{M}$ are replaced by the spaces $H^1_{\cP^\al,c}(M)$ similarly defined in terms of the subordinated semigroup $\{ \e^{-t \cL^\al} \}$ for any $\alpha \in (0,1)$; the case of the Poisson semigroup corresponds to the choice $\alpha=1/2$. For example, for any manifold $M$ in the class $\cM$, one can prove the proper inclusions
\[
H_{\cP^\al,\ga/\al+\vep}^1(M) \subset \gXga{M} \subset H_{\cP^\al,\ga/\al-\vep}^1(M)
\]
for all $\ga,\vep>0$, and also show that $H_{\cP^\al,\ga/\al}^1(S) \neq \gXga{S}$ when $S$ is a Damek--Ricci space. 
Indeed, the analogue of Proposition \ref{p: H1Hloc} holds for an arbitrary subordinated semigroup (see \cite[Corollary 5.4]{MaMV2}), while the estimates in Lemma \ref{l: est Lpt} and Proposition \ref{p: maximal Poisson H3} apply to any $\alpha \in (0,1)$. We leave the remaining easy details to the interested reader.
\end{remark}


\begin{thebibliography}{HLMMY}

\bibitem[A]{A1}
J.-Ph. Anker,
Sharp estimates for some functions of the Laplacian on noncompact symmetric spaces,
\textit{Duke Math. J.} \textbf{65} (1992), 257--297.

\bibitem[ADY]{ADY}
J.-Ph. Anker, E. Damek and C. Yacoub,
Spherical analysis on harmonic $AN$ groups,
\textit{Ann. Scuola Norm. Sup. Pisa Cl. Sci. (4)} \textbf{23} (1996), 643--679.

\bibitem[As]{As}
F. Astengo,
A class of $L^p$ convolutors on harmonic extensions of H-type groups,
\textit{J. Lie Theory} \textbf{5} (1995), no. 2, 147--164.

\bibitem[AMR]{AMR}
P. Auscher, A. McIntosh and E. Russ,
Hardy spaces of differential forms on Riemannian manifolds,
\textit{J. Geom. Anal.} \textbf{18} (2008), 192--248.

\bibitem[Bu]{Bu}
P. Buser,
A note on the isoperimetric constant,
\textit{Ann. Sci. \'Ecole Norm. Sup (4)} \textbf{15} (1982), no.~2, 213--230.

\bibitem[CMM1]{CMM1}
A. Carbonaro, G. Mauceri and S. Meda, 
$H^1$ and $BMO$ for certain locally doubling metric measure spaces,
\textit{Ann. Scuola Norm. Sup. Pisa Cl. Sci. (5)} \textbf{8} (2009), 543--582.

\bibitem[CMM2]{CMM2}
A. Carbonaro, G. Mauceri and S. Meda, 
$H^1$ and $BMO$ for certain locally doubling metric measure spaces of finite measure,
\textit{Colloq. Math.} \textbf{118} (2010), 13--41.

\bibitem[CeM]{CM}
D. Celotto and S. Meda, 
On the analogue of the Fefferman--Stein theorem on graphs with the Cheeger property,
\textit{Ann. Mat. Pura Appl.} \textbf{197} (2018), 1637--1677.

\bibitem[Ch]{Ch}
I. Chavel,
\textit{Isoperimetric Inequalities},
Cambridge Tracts in Mathematics, vol. 145, Cambridge University Press, Cambridge, 2001.

\bibitem[CMa]{CMa}
M. G. Cowling and A. Martini,
Sub-Finsler geometry and finite propagation speed,
in: \textit{Trends in Harmonic Analysis}, Springer INDAM series, vol.~3, Springer, Milan, 2013, pp.~147--205.

\bibitem[DR]{DR}
E. Damek and F. Ricci,
Harmonic analysis on solvable extensions of H-type groups,
\textit{J. Geom. Anal.} \textbf{2} (1992), 213--248.

\bibitem[DKKP]{DKKP}
S. Dekel, G. Kerkyacharian, G. Kyriazis and P. Petrushev,
Hardy spaces associated with non-negative self-adjoint operators,
\textit{Studia Math.} \textbf{239} (2017), 17--54. 

\bibitem[DiB]{DiB}
B. Di Blasio,
Paley--Wiener type theorems on harmonic extensions of H-type groups, 
\textit{Monatsh. Math.} \textbf{123} (1997), no. 1, 21--42. 

\bibitem[FS]{FS}
C. Fefferman and E.~M.~Stein,
Hardy spaces of several variables, 
\textit{Acta Math.} \textbf{129} (1972), 137--193.

\bibitem[Go]{Go}
D. Goldberg,
A local version of real Hardy spaces,
\textit{Duke Math. J.} \textbf{46} (1979), 27--42.

\bibitem[Haa]{Haa}
M.~Haase, 
\textit{The Functional Calculus for Sectorial Operators},
Operator Theory: Advances and Applications, No. 169, Birkh\"auser-Verlag, Basel, 2006.

\bibitem[He]{He}
E. Hebey,
\textit{Sobolev Spaces on Riemannian Manifolds},
Lecture Notes in Mathematics, vol. 1635, Springer-Verlag, Berlin, 1996.

\bibitem[HLMMY]{HLMMY}
S. Hofmann, G. Lu, D. Mitrea, M. Mitrea and L. Yan,
Hardy spaces associated to non-negative self-adjoint operators satisfying Davies–Gaffney estimates,
\textit{Mem. Amer. Math. Soc.} \textbf{214} (2011), no.\ 1007.

\bibitem[I1]{I2}
A. D. Ionescu,
Fourier integral operators on noncompact symmetric spaces of real rank one,
\textit{J. Funct. Anal.} \textbf{174} (2000), 274--300.

\bibitem[I2]{I}
A. D. Ionescu,
Singular integrals on symmetric spaces. II,
\textit{Trans. Amer. Math. Soc.} \textbf{355} (2003), 3359--3378.

\bibitem[Lo]{Lo}
N. Lohou\'e,
Transform\'ees de Riesz et fonctions sommables,
\textit{Amer. J. Math.} \textbf{114} (1992), 875--922. 

\bibitem[MaMV1]{MaMV1}
A. Martini, S. Meda and M. Vallarino,
A family of Hardy type spaces on nondoubling manifolds,
\textit{Ann. Mat. Pura Appl.} \textbf{199} (2020), 2061--2085.

\bibitem[MaMV2]{MaMV2}
A. Martini, S. Meda and M. Vallarino,
Maximal characterisation of local Hardy spaces on locally doubling manifolds,  
\textit{Math. Z.} \textbf{300} (2022), 1705--1739.

\bibitem[MM]{MM} G. Mauceri and S. Meda,
$BMO$ and $H^1$ for the Ornstein--Uhlenbeck operator,
\textit{J. Funct. Anal.} \textbf{252} (2007), 278--313.

\bibitem[MMV1]{MMV1}
G. Mauceri, S. Meda and M. Vallarino,
Hardy-type spaces on certain noncompact manifolds and applications,
\textit{J. London Math. Soc. (2)} \textbf{84} (2011), 243--268.

\bibitem[MMV2]{MMV2}
G. Mauceri, S. Meda and M. Vallarino,
Atomic decomposition of Hardy type spaces on certain noncompact manifolds,
\textit{J. Geom. Anal.} \textbf{22} (2012), 864--891.

\bibitem[MMV3]{MMV3}
G. Mauceri, S. Meda and M. Vallarino, 
Harmonic Bergman spaces, the Poisson equation and the dual of Hardy type spaces on certain noncompact manifolds,
\textit{Ann. Scuola Norm. Sup. Pisa Cl. Sci. (5)} \textbf{14} (2015), no. 4, 1157--1188.  

\bibitem[MMV4]{MMV4}
G. Mauceri, S. Meda and M. Vallarino,
Endpoint results for spherical multipliers on noncompact symmetric spaces,
\textit{New York J. Math.} \textbf{23} (2017), 1327--1356.

\bibitem[MMV5]{MMV5}
G. Mauceri, S. Meda and M. Vallarino,
Higher order Riesz transforms on noncompact symmetric spaces,
\textit{J. Lie Theory} \textbf{28} (2018), 479--497.

\bibitem[MVe]{MVe}
S. Meda and G. Veronelli,
Local Riesz transform and local Hardy type spaces on Riemannian manifolds with bounded geometry,
\textit{J. Geom. Anal.} \textbf{32} (2022), art. no. 55.

\bibitem[MVo]{MVo}
S. Meda and S. Volpi,
Spaces of Goldberg type on certain measured metric spaces,
\textit{Ann. Mat. Pura Appl.} \textbf{196} (2017), 947--981.

\bibitem[PS]{PS}
N. Peyerimhoff and E. Samiou,
Spherical spectral synthesis and two-radius theorems on Damek--Ricci spaces,
\textit{Ark. Mat.} \textbf{48} (2010), no. 1, 131--147.

\bibitem[Ru]{Ru}
E.~Russ,
$H^1$--$L^1$ boundedness of Riesz transforms on Riemannian manifolds and on graphs,
\textit{Potential Anal.} \textbf{14} (2001), 301--330.

\bibitem[St]{St1}
E. M. Stein,
\textit{Topics in Harmonic Analysis Related to the Littlewood--Paley Theory},
Annals of Math. Studies, No. 63, Princeton, N. J., 1970.

\bibitem[Str]{strichartz}
R. S. Strichartz,
Analysis of the Laplacian on the complete Riemannian manifold,
\textit{J. Funct. Anal.} \textbf{52} (1983), no. 1, 48--79.

\bibitem[T]{T}
M. E. Taylor,
Hardy spaces and bmo on manifolds with bounded geometry, 
\textit{J. Geom. Anal.} \textbf{19} (2009), 137--190.

\bibitem[VSC]{VSC}
N. Varopoulos, L. Saloff-Coste and T. Coulhon,
\textit{Analysis and Geometry on Groups},
Cambridge University Press, Cambridge, 1992.

\bibitem[Y]{Y}
K. Yosida,
\textit{Functional Analysis},
4\textsuperscript{th} ed., Springer-Verlag, Berlin, Heidelberg and New York, 1974.  

\end{thebibliography}
\end{document}